\documentclass[12pt,a4paper,leqno]{article}
\usepackage{amsmath,amssymb,amsthm}
\usepackage{color}
\usepackage[a4paper,margin=2.5cm]{geometry}
\usepackage[colorlinks]{hyperref}
\allowdisplaybreaks
\newtheorem{theorem}{Theorem}[section]
\newtheorem{corollary}[theorem]{Corollary}
\newtheorem{lemma}[theorem]{Lemma}
\newtheorem{proposition}[theorem]{Proposition}
\theoremstyle{definition}
\newtheorem{definition}[theorem]{Definition}
\newtheorem{remark}[theorem]{Remark}

\newtheorem{example}[theorem]{Example}

\numberwithin{equation}{section}
\newcommand{\R}{\mathbb{R}}
\newcommand{\N}{\mathbb{N}}
\newcommand{\eps}{\varepsilon}
\newcommand{\ab}{[a,b]}
\newcommand{\OT}{\,[0,T]}
\newcommand{\wt}{\widetilde}
\newcommand{\var}{\mbox{var}}
\newcommand{\uu}{\accent23 u}

\begin{document}

\title{Bifurcation of periodic solutions to nonlinear measure differential
equations.
}

\author{C.~Mesquita\footnote{Department of Mathematics, Federal University
of S\~ao Carlos, Caixa Postal ~676, 13565.905 S\~ao Carlos SP, Brazil.
E-mail: {\tt mc12stefani@hotmail.com}},
M.~Tvrd\'{y}
\thanks{Institute of Mathematics, Czech Academy of Sciences, Prague.
E-mail: {\tt tvrdy@math.cas.cz}}}

\date{\today}
\maketitle

\begin{abstract}
This paper is devoted to bifurcations of periodic solutions of nonlinear measure
differential equations with a parameter. Main tools are nonlinear generalized
differential equations (in the sense of Kurzweil) and the Kurzweil gauge type
generalized integral. We continue the research started in \cite{carol} and
\cite{JC}.
\end{abstract}

\section{Introduction}

The concept of measure differential equations arose more or less together
with the concepts of impulse systems or distributional differential equations.
They generally try to describe some physical or biological problems, such as
heartbeat, blood flow, pulse/frequency modulated systems, and/or models for
biological neural networks. In these models, derivatives are understood in the
sense of distributions and the solutions are generally discontinuous, but not
too bad from another point of view, i.e. they are usually regulated or have
bounded variation. Early results were summarized e.g. in monographs \cite{PaDe},
\cite{SaPe}, \cite{BaSi} and references therein.

The motivation for studying such problems was also some models created in
control theory, in which it turned out that measures can be more suitable
controls, cf. e.g. \cite{MiRu}. Moreover, differential equations with measure
also appear in non-smooth mechanics, cf. \cite{Br}. More recent references
are e.g. \cite{CaSu}, \cite{CaSu2}, \cite{Sa}, \cite{RaTo} and many others.

In this article we consider the measure differential system
\begin{equation}\tag{1.1}
   Dx=f(\lambda,x,t)+g(x,t).D h,
\end{equation}
where $D$ stands for the distributional derivatives and $\lambda$ is
a parameter. The assumptions about the functions $f,\,g,$ measure $D h$ as well
as the exact definition of the solution (in general, these will be vector-valued
regulated functions) will be given later. We are particularly interested
in bifurcations with respect to a given periodic solution. To this end, an
important tool are generalized ordinary differential equations (we write simply
generalized ODEs). These equations were introduced in the middle of the 1950s
by Kurzweil in \cite{Ku57, Ku58}. Since then, many authors have dealt with
the potentialities of this theory (see e.g. \cite{BFM, Kurz, Schwabik, MST} and
references therein). In \cite{JC} the authors introduced the concept of
bifurcation point with respect to the trivial solution of the periodic problem
for the generalized ODE
\begin{equation}\tag{1.2}
   \dfrac{dx}{d\tau}=DF(\lambda,x,t),
\end{equation}
where $T>0,$ $F:\Lambda\times\Omega\times\OT\rightarrow\R^n,$ $\Lambda\subset\R$
and $\Omega\subset\R^n.$  By means of the coincidence degree theory, they
established conditions sufficient for the existence of such a~bifurcation point,
cf. \cite[Theorem 5.6]{JC}. Similar questions have been already studied in the
thesis \cite{carol}.

In particular, we will show that, under reasonable assumptions, our measure
differential system \thetag{1.1} becomes a special case of equations of the
form \thetag{1.2}. Thus, for the periodic problem for \thetag{1.1} we obtain
the existence of its bifurcation point as a direct corollary of the analogous
result from \cite{JC}. Furthermore, we will present conditions necessary
for the existence of the bifurcation point of the periodic problem for \thetag{1.2}
and apply this result to \thetag{1.1}.

\section{Preliminaries (Kurzweil integral and generalized ODEs)}\label{S2}

One of our main tools are the Kurzweil integral and its special case,
Kurzweil-Stieltjes integral. This kind of integral has been introduced
by Kurzweil in the middle of the fifties, cf. \cite{Ku57, Ku58}. In this
section, we summarize some of its basic concepts needed later.

Throughout the paper, the symbol $X$ stands for a Banach space equipped
with the norm $\|\cdot\|_X.$ Usually we restrict ourselves to the cases
$X=R^n$ or $X=\mathcal{L}(\R^n),$ where $\mathcal{L}(\R^n)$ is the space
of real $n\times n$-matrices equipped with the norm
\[
   \|A\|_{n\times n}=\max_{i\in\{1,\dots n\}}\sum_{j=1}^n |a_{i,j}|
   \quad\mbox{for \ } A=(a_{i,j})_{i,j\in\{1,\dots n\}}\in\mathcal{L}(\R^n).
\]
and and $\R^n$ is the space of real $n\times 1$-matrices equipped
with the norm
\[
   \|x\|_n=\sum_{j=1}^n |x_i|
   \quad\mbox{for \ } x=(x_i)_{i\in\{1,\dots n\}}\in\R^n.
\]

The function $x:\ab\to X$ is {\em regulated}, if the lateral limits
\[
   x(t-)=\lim_{\tau\to t-}x(\tau) \quad\mbox{and}\quad
   x(s+)=\lim_{\tau\to s+}x(\tau)
\]
exist for all $t\in(a,b]$ and $s\in [a,b).$ The space of functions
$x:\ab\to X$ which are regulated on $\ab$ will be denoted as $G(\ab;X).$
As usual, $\Delta^+x(t)=x(t+)-x(t)$ and $\Delta^-x(t)=x(t)-x(t-)$
whenever the expressions on the right sides have a sense.
It is well known that, when equipped with the supremal norm
$\|x\|_{\infty}=\sup_{t\in\ab}\|x(t)\|_n,$ $G(\ab;X)$ is a~Banach space
(see e.g. \cite{Hoenig}). As usual, $BV(\ab;X)$ stand for
the space of functions $x:\ab\to X$ having a bounded variation on
$\ab$ and $\var_a^b f$ is the variation of the function $f$ over $\ab.$
If $X=\R^n,$ we write simply $G\ab$ and $BV\ab$ instead of $G(\ab;X)$
and $BV(\ab;X),$ respectively.

In this paper, by an integral we mean the integral introduced by J.~Kurzweil
in \cite{Ku57}. Its definition relies on the notions of gauges and tagged
partitions fine with respect to the gauges:

Let $\ab$ be a~bounded closed interval. Finite collections of point-interval
pairs $P\,{=}\,(\tau_j,[t_{j-1},t_j])_{j=1}^{\nu(P)}$ such that
$a=t_0\le t_1\le\cdots\le t_{\nu(P)}=b$ and $\tau_j\in[t_{j-1},t_j]$ for
$j\in\{1,\dots,n\}$ are called {\em tagged partitions} of $\ab.$ Furthermore,
any positive function $\delta:\ab\to(0,\infty)$ is called a~\emph{gauge}
on $\ab.$ Given a~gauge $\delta$ on $\ab$, the partition
$P=(\tau_j,[t_{j-1},t_j])_{j=1}^{\nu(P)}$ is called $\delta$-{\em fine} if
\[
 [\alpha_{j{-}1},\alpha_j]
 \subset(\tau_j\,{-}\,\delta(\tau_j),\tau_j\,{+}\,\delta(\tau_j))
 \quad\mbox{for all \ } j\in\{1,2,\dots,\nu(P)\}.
\]
Recall, that by Cousin Lemma \cite{cou} (see also
e.g.~\cite[Lemma 1.4]{Schwabik} or \cite[Lemma 6.2.3]{MST}) there always
exists a~$\delta$-fine tagged partition of $\ab$ for any $\delta$ on $\ab.$

\begin{definition}
Let $-\infty<a<b<\infty$ and let $X$ be a Banach space. Then the function
$U:\ab\times\ab\to X$ is said to be {\em Kurzweil integrable} on $\ab$ if
there is an $I\in X$ such that for every $\eps>0$ we can find a~gauge
$\delta$ on $\ab$ such that
\[
 \left\|\sum_{j=1}^{\nu(P)}[U(\tau_j,t_j)-U(\tau_j,t_{j-1})]-I\right\|_X
 <\eps.
\]
holds for every $\delta$-fine tagged partition
$P=(\tau_j,[t_{j-1},t_j])_{j=1}^{\nu(P)}$ of $\ab.$

In such a~case, $I$ is said to be the Kurzweil integral of $U$ over $\ab$
and we write
\[
   I=\int_a^b DU(\tau,t).
\]
If the integral $\int_a^b DU(\tau,t)$ has a~sense, we put
\[
    \int_b^a DU(\tau,t)=-\int_a^b DU(\tau,t).
\]
Furthermore,
\[
  \int_a^b DU(\tau,t)=0 \quad\mbox{if \ } a=b.
\]
\end{definition}

\begin{remark}
If $U(\tau,t)\,{=}\,G(\tau)\,H(t),$ where
$G\,{:}\,\ab\,{\to}\,\mathcal{L}(\R^n)$ and $H\,{:}\,\ab\,{\to}\,\R^n$ then
the integral $\int_a^b DU(\tau,t)$ reduces to the Kurz\-weil-Stieltjes
integral $\int_a^b G\,d H.$ Similarly, if $U(\tau,t)=H(t)\,G(\tau),$ where
$H:\ab\to\mathcal{L}(\R^n)$ and $G:\ab\to\R^n,$ then
\[
   \int_a^b DU(\tau,t)=\int_a^b d H\,G.
\]
Both these cases were considered in details in \cite{MST}. Finally,
if $H(t)\equiv t,$ the integral is known as the Henstock-Kurzweil
integral.
\end{remark}

The first part of the following assertion follows from
\cite[Corollary 14.18]{Kurz}. The second one follows directly from
the definition of the Kurzweil integral.

\begin{lemma}\label{Ant}
Let $U:\ab\times\ab\to X$ be Kurzweil integrable and regulated in
the second variable on $\ab$ and
\[
   v(s)=\int_a^s DU(\tau,t) \quad\mbox{for\ } s\in\ab.
\]
Then $v$ is regulated on $\ab,$
\[
  \Delta^-v(t)=U(t,t)-U(t,t-) \mbox{\ if\ } t\in[a,b)
  \mbox{\ and\ }
  \Delta^+v(t)=U(t,t+)-U(t,t) \mbox{\ if\ } t\in(a,b].
\]
Moreover, if there are functions $f\,{:}\,\ab\,{\to}\,\R$
regulated on $\ab$ and $g\,{:}\,\ab\,{\to}\,\R$ nondecreasing on $\ab$
and such that
\[
  \|U(\tau,t)-U(\tau,s)\|_X\le |f(\tau)|\,|g(t)-g(s)|
  \quad\mbox{for all \ } t,s,\tau\in\ab,
\]
then
\[
   \left\|\int_0^s DU(\tau,t)\right\|_X\le\int_0^s |f(\tau)|\,d g(\tau)
   \quad\mbox{for all \ } s\in\ab.
\]
\end{lemma}

\medskip

Now, we will recall the concept of a~solution to the generalized ODE
\begin{equation}\label{EDOG}
   \dfrac{dx}{d\tau}=DF(x,t).
\end{equation}

\begin{definition}\label{solucao}
Let $\Omega\subset X$ be open and let $F:\Omega\times\ab\to X.$
Then the function $x:\ab\to X$ is said to be a~{\em solution}
of the {\em generalized ODE} \eqref{EDOG} on $\ab$ whenever
\[
   x(s)\in\Omega\mbox{ \ and \ }
   x(s)=x(a)+\int_a^s DF(x(\tau),t) \quad\mbox{for all \ } s\in\ab.
\]
\end{definition}

A~proper class of right-hand sides of equation \eqref{EDOG} is given
by the following definition.

\begin{definition}\label{ClassF}
Let $h:\ab\to\R$ be nondecreasing on $\ab,$ let $\omega:[0,\infty)\to\R$
be increasing and continuous on $[0,\infty)$ with $\omega(0)=0$ and let
$\Omega\subset X$ be open. Then ${\mathcal F}(\Omega\times\ab,h,\omega;X)$
is the set of all functions $F:\Omega\times\ab\to X$ fulfilling the
relations
\begin{align}\label{Fa}
   &\|F(x,t_2)-F(x,t_1)\|_X\le|h(t_2)-h(t_1)|
\\\noalign{\noindent\mbox{and}}\label{Fa1}
   &\|F(x,t_2)\,{-}\,F(x,t_1)\,{-}\,F(y,t_2)\,{+}\,F(y,t_1)\|_X
   \,{\le}\,\omega(\|x{-}y\|_X)|h(t_2){-}h(t_1)|
\\[1mm]&\quad\nonumber
   \mbox{\ for all\ } x,y\in\Omega \mbox{\ and\ } t_1,t_2\in\ab.
\end{align}

If $X=\R^n,$ we write ${\mathcal F}(\Omega\times\ab,h,\omega)$
instead of ${\mathcal F}(\Omega\times\ab,h,\omega;\R^n).$
\end{definition}

\medskip
Next result is a slightly modificated version of \cite[Lemma 5]{JA}.
In the proof, one have to take into mind that a composition of a continuous
function with a regulated one is always regulated.

\begin{lemma}\label{JA}
Let $F\in{\mathcal F}(\Omega\times\ab,h,\omega),$ where $h:\ab\to\R$ is
nondecreasing on $\ab,$ $\omega:[0,\infty)\to\R$ is increasing and
continuous on $[0,\infty),$ $\omega(0)=0$ and $\Omega\subset\R^n$ is open.
Then
\[
    \left\|\int_{s_1}^{s_2} D[F(x(\tau),t)-F(y(\tau),t)]\right\|_n
    \le\int_{s_1}^{s_2}\omega(\|x(t)-y(t)\|_n)\,d h(t)
\]
for all $[s_1,s_2]\subset\ab$ and $x,y\in G\ab$ such that
$x(t)\in\Omega$ and $y(t)\in\Omega$ for all $t\in [a,b].$
\end{lemma}

Next assertion is Lemma 4.5 from \cite{BFM}, for finite dimensional
case see e.g. Lemma 3.9 and Corollary 3.15 in \cite{Schwabik}.

\begin{lemma}\label{rrr}
Assume that $F:\Omega\times\ab\to X$ fulfils \eqref{Fa}. Then, for
any \ $x\in G\ab$ such that $x(s)\in\Omega$ for all $s\in\ab,$ the
integral $\int_{a}^{b}DF(x(\tau),t)$ exists and the inequality
\[
   \left\|\int_{s_1}^{s_2}DF(x(\tau),t)\right\|_n\le |h(s_2)-h(s_1)|
\]
is true for all $s_1,s_2\in\ab.$ Furthermore, the function
\[
   s\in\ab\to\int_a^s DF(x(\tau),t)
\]
has a~bounded variation on $\ab.$

Finally, every solution $x$ of \eqref{EDOG} has a bounded variation
on $\ab$ and, in particular, it is regulated on $\ab.$
\end{lemma}

\begin{remark}
If we consider in \eqref{EDOG} a~particular case $F(x,t)\,{=}\,A(t)\,x,$
where $A\,{:}\,\ab\to\mathcal{L}(\R^n),$ we obtain the generalized
linear ODE
\begin{equation}\label{linear}
   \dfrac{dx}{d\tau}=D[A(t)\,x]
\end{equation}
Obviously, the function $x:\ab\to\R^n$ is a~solution of the generalized
linear ODE \eqref{linear} on $\ab,$ whenever
\begin{equation}\label{linear1}
   x(s)-x(0)=\int_{0}^{s} d[A(t)]\,x(t) \quad\mbox{for \ } s\in\ab,
\end{equation}
where the integral stands for the Kurzweil-Stieltjes one.
\end{remark}

Finally, we state the following basic result from \cite[Theorem 5.1]{BFM}
well illustrating the importance of the class
${\mathcal F}(\Omega\times\ab,h,\omega)$ in the theory of generalized ODEs.
For the finite dimensional case, see \cite[Theorem 4.2]{Schwabik}.

\begin{theorem}\label{ExSol}
Assume there are $h:\ab\to\R$ nondecreasing on $\ab$ and \
$\omega:[0,\infty)\,{\to}\,\R$ increasing and continuous on $[0,\infty)$
with $\omega(0)=0$ such that $F\in\mathcal F(\Omega{\times}\OT,h,\omega;X).$
Furthermore, let $(x_0,t_0)\in\Omega\times[a,b)$ be such that
$x_0+F(x_0,t_0+)-F(x_0,t_0)\in\Omega.$ Then there is a $\Delta>0$ such that
the equation \eqref{EDOG} has a solution $x$ on $[t_0,t_0+\Delta]$ such that
$x(t_0)=x_0.$
\end{theorem}

\section{Bifurcation theory for generalized ODEs}\label{necessary}

In this section, we will consider the concept of a~bifurcation point with
respect to a given solution of the parameterized periodic boundary value
problem for the periodic problem
\begin{equation}\label{GDEper}
   \dfrac{dx}{d\tau}=DF(\lambda,x,t),\quad x(0)=x(T).
\end{equation}

In the rest of the paper we have $a=0$ and $0<b=T<\infty.$ Furthermore, given
a Banach space $X,$ the symbol $Id$ stands for identity operator on $X$ and,
for a~given $x_0\in X$ and $\rho>0,$ we denote by $B(x_0,\rho)$ the closed
ball in $X$ centered at $x_0$ and with the radius~$\rho.$

\begin{definition}\label{sol}
Let $\Omega\subset\R^n$ and $\Lambda\subset\R$ be open and
$F:\Lambda\times\Omega\times\OT\to\R^n.$ Then the couple
$(x,\lambda)\in G\OT\times\Lambda$ is a~{\em solution} of the problem
\eqref{GDEper} whenever
\[
    x(s)\in\Omega \mbox{ \ and \ }
    x(s)=x(0)+\int_{0}^{s}DF(\lambda,x(\tau),t)\quad\mbox{for \ } s\in[0,T],
\]
and \ $x(0)=x(T).$
\end{definition}

For our purposes, the following hypotheses will be helpful.
\begin{equation}\label{B1}\hskip-5mm
\left\{\begin{array}{l} \Omega\,{\subset}\,\R^n \mbox{\ and\ }
   \Lambda\,{\subset}\,\R \mbox{\ are open sets};\,\,
   F:\Lambda\,{\times}\,\Omega\,{\times}\,\OT\to\R^n
   \mbox{\ and}
  \\[2mm]
   \mbox{there are\ } h\,{:}\,\OT\,{\to}\,\R \mbox{\ nondecreasing}
   \mbox{\ and\ }\omega\,{:}\,[0,\infty)\,{\to}\,[0,\infty)
  \\[1mm]
   \mbox{increasing and continuous and such that \ } \omega(0)=0 \mbox{ \ and}
  \\[1mm]
   F(\lambda,\cdot,\cdot)\in\mathcal{F}(\Omega\times\OT,h,\omega)
   \mbox{\ for each\ } \lambda\in\Lambda;
\end{array}\right.
\end{equation}
\begin{equation}\label{B2}\hskip-2mm
\left\{\begin{array}{l}
   (x_0,\lambda)\in G\OT\times\Lambda \mbox{\ is a~solution of \eqref{GDEper}
   for any\ } \lambda\in\Lambda \mbox{\ and}
   \\[1mm]
   \mbox{there is\ } \rho>0 \mbox{\ such that\ } x(t)\in\Omega
   \mbox{\ for all\ } (t,x)\in\OT\times B(x_0,\rho).
\end{array}\right.
\end{equation}
Furthermore, let us define
\begin{equation}\label{operador}\hskip-47mm
\left\{\begin{array}{l}\displaystyle
    \Phi(\lambda,x)(s)=x(T)+\int_0^s DF(\lambda,x(\tau),t)
   \\[3mm]
    \quad\mbox{for\ } \lambda\,{\in}\,\Lambda,\,x\,{\in}\,B(x_0,\rho)
    \mbox{\ and\ } s\,{\in}\,\OT,
\end{array}\right.
\end{equation}
whenever the Kurzweil integral on the right hand side has a~sense.

\begin{proposition}\label{equiv}
Assume \eqref{B1} and \eqref{B2} and let the operator $\Phi$ be defined by
\eqref{operador}. Then $\Phi(\lambda,\cdot)$ maps $B(x_0,\rho)$ into $G\OT$
for any $\lambda\in\Lambda.$ Moreover, problem \eqref{GDEper} is equivalent
to finding solutions $(x,\lambda)$ of the operator equation
\begin{equation}\label{op-eq}
   x=\Phi(\lambda,x).
\end{equation}
\end{proposition}
\begin{proof}
The first part of the statement follows from Lemma \ref{rrr}. Furthermore, if
\begin{equation}\label{fixed-eq}
   x(s)=x(T)+\int_0^s DF(\lambda,x(\tau),t) \quad\mbox{for \ } s\in\OT,
\end{equation}
then for $s=0$ we get $x(0)=x(T).$ As a result, $(x,\lambda)$ is a solution
to \eqref{GDEper}. The opposite implication is obvious.
\end{proof}

Let use recall that recently Federson, Mawhin and Mesquita extended some
classical conditions on the existence of a periodic solution of nonautonomous
ordinary differential equations to the problem of the form \eqref{fixed-eq}
in \cite{JC}, cf. sec.4 therein.

In general, a bifurcation occurs whenever a small change of the parameters
of the given problem causes a qualitative change of the behavior of its
solutions. In our case, we understand to this phenomena in the following way.

\begin{definition}\label{bif}
Solution $(x_0,\lambda_0)\in G\OT\times\Lambda$ of \eqref{op-eq} is said to
be a~{\em bifurcation point} of \eqref{op-eq} (i.e. of \eqref{GDEper}) if
every neighborhood of $(x_0,\lambda_0)$ in $B(x_0,\rho)\times\Lambda$
contains a~solution $(x,\lambda)$ of \eqref{op-eq} such that $x\ne x_0.$
\end{definition}

By an obvious modification of the proof of \cite[Theorem 5.6]{JC}, providing
conditions sufficient for the existence of a bifurcation point of \eqref{op-eq},
we can state its slightly reformulated version.

As usual (cf. e.g. \cite[Section 5.2]{DM}), for a Banach space $X,$ open bounded
set $\Omega\subset X,$ a compact operator $\Phi:\overline{\Omega}\to X$ and
$z\notin(I-\Phi)(\partial\,\Omega),$ the symbol $\mbox{deg}_{LS}(Id-\Phi,\Omega,z)$
stands for the {\it Leray-Schauder degree} of $Id-\Phi$ with respect to $\Omega,$ at
the point $z.$ Furthermore, if $a$ is an isolated fixed point of $\Phi,$ then
the value $\mbox{ind}_{LS}(Id-\Phi,a)$ defined by
\[
    \mbox{ind}_{LS}(Id-\Phi,a)=\deg_{LS}[I-\Phi,B(a,r),0]
    \quad\mbox{for small \ } r>0
\]
is said to be the {\it Leray-Schauder index} of $Id-\Phi$ at $a$, or sometimes also
the {\it index of an isolated fixed point} of $\Phi.$

\begin{theorem}\label{TheEx}
Assume \eqref{B1}, \eqref{B2} and
\begin{equation}\label{B3}\hskip-10mm
\left\{\begin{array}{l}
   \mbox{there is a function \ } \gamma:\OT\to\R \mbox{\ nondecreasing and
   such that}
 \\[1mm]
   \mbox{for any\ } \eps>0 \mbox{\ there is a\ } \delta>0 \mbox{\ such that}
\\[2mm]\displaystyle\quad
   \|F(\lambda_1,x,t)\,{-}\,F(\lambda_2,x,t)
                     \,{-}\,F(\lambda_1,x,s)\,{+}\,F(\lambda_2,x,s)\|_n
   {<}\,\eps\,|\gamma(t)\,{-}\,\gamma(s)|
 \\[1mm]
   \mbox{for\ } x\in\Omega,\,t,s\in\OT
   \mbox{\ and\ } \lambda_1,\lambda_2\in\Lambda
   \mbox{\ such that\ } |\lambda_1-\lambda_2|<\delta.
\end{array}\right.\hskip-8mm
\end{equation}
Moreover, let the operator $\Phi$ be defined by \eqref{operador} and let
$[\lambda^*_1,\lambda^*_2]\subset\Lambda$ be such that
\begin{align}\label{ind}
  &x_0 \mbox{\ is an isolated fixed point of the operators \ }
  \Phi(\lambda^*_1,\cdot) \mbox{ \ and \ } \Phi(\lambda^*_2,\cdot)
\\\noalign{\noindent\mbox{and}}
  &\mbox{ind}_{LS}(Id-\Phi(\lambda^*_1,\cdot),0)
    \ne\\mbox{ind}_{LS}(Id-\Phi(\lambda^*_2,\cdot),0).
\end{align}
Then there is \ $\lambda_0\in[\lambda^*_1,\lambda^*_2]$ \ such that
\ $(x_0,\lambda_0)$ is a bifurcation point of \eqref{GDEper}.
\end{theorem}

Our wish is to deliver also conditions which are necessary for the existence of
a~bifurcation point of the equation \eqref{op-eq}. This will be given by
Theorem \ref{NC}. Before formulating and proving this theorem let us take attention
to the following immediate observation:

If $(x_0,\lambda_0)$ is a solution to \eqref{op-eq}, then, by Definition \ref{bif}
it is not a bifurcation point of \eqref{op-eq} whenever it has a neighborhood
$\mathcal{U}\subset B(x_0,\rho)\times\Lambda$ in $G\OT\times\R$ such that $x=x_0$
holds for any solution $(x,\lambda)$ to \eqref{op-eq} belonging to $\mathcal{U}.$
It follows that the set of couples $(x,\lambda)\in G\OT\times\Lambda$ which are not
bifurcation points of \eqref{op-eq} is open in $G\OT\times\R.$ In particular, we have

\begin{corollary}\label{cor-def}
If $(x_0,\lambda_0)$ is not a bifurcation point of \eqref{op-eq} then there is
a $\delta>0$ such that the set $B((x_0,\lambda_0),\delta)$ does not contain any
bifurcation point of \eqref{op-eq}.
\end{corollary}

Furthermore, in the proof of Theorem \ref{NC} the notion of the derivative of the
operator function $\Phi$ is needed.

\begin{definition}\label{frechet}
Let $X,Y$ be Banach spaces, $D\subset X$ open and $G$ an operator function mapping
$D$ into $Y.$ By the derivative $G\sp{\,\prime}(x)$ of $G$ at the point $x\in D$ we
understand its Frechet derivative at $x,$ i.e. $G\sp{\,\prime}(x)$ is the linear
bounded operator on $X$ such that
\[
 \lim_{\vartheta\to 0+}
   \left\|\frac{G(x+\vartheta\,z)-G(x)}{\vartheta}-G\sp{\,\prime}(x)\,z\right\|_Y=0
 \quad\mbox{for all\ } z\in X.
\]
In particular, derivative of $\Phi(\lambda,\cdot)$ at $x$ will be denoted by
$\Phi\sp{\,\prime}_x(\lambda,x)$ and, similarly, derivative of the function
$F(\lambda,\cdot,t):\Omega\to\R^n$ at $x\in\Omega$ is denoted as
$F\sp{\,\prime}_x(\lambda,t,x).$ Recall that
$F\sp{\,\prime}_x(\lambda,t,x)\in\mathcal{L}(\R^n)$ is represented by
$n\times n$-Matrix.
\end{definition}

Next assertion provides the explicit form of the derivative of
$\Phi(\lambda,\cdot).$

\begin{proposition}\label{derivative}
Assume that the conditions \eqref{B1} and \eqref{B2} are satisfied, $\Phi$ is
defined by \eqref{operador} and $\rho>0$ be given by \eqref{B2}.
Furthermore, suppose that for each \ $(\lambda,x,t)\in\Lambda\time\Omega\times\OT$
the function \ $F$ \ has a derivative $F\sp{\,\prime}_x(\lambda,x,t)$ \ which is
for each $(\lambda,t)\in\Lambda\times\OT$ continuous with respect to $x$ on $\Omega$
and such that
\begin{equation}\label{C1}
\left\{\begin{array}{l}
  \quad
  F\sp{\,\prime}_x(\lambda,\cdot,\cdot)
   \in\mathcal{F}(\Omega\times\OT,\wt{h},\wt{\omega}\,;\mathcal{L}(\R^n))
  \mbox{\ for all\ } \lambda\in\Lambda,
 \\[1mm]\mbox{where\ }
 \\[1mm]\quad
  \wt{h}:\OT\to[0,\infty) \mbox{\ is nondecreasing on\ } \ab
 \\[1mm]\mbox{and}
 \\[1mm]\quad
  \wt{\omega}\,{:}\,[0,\infty)\,{\to}[0,\infty)
  \mbox{\ is continuous and increasing on\ } [0,\infty)
  \mbox{\ and \ } \wt{\omega}(0)=0.
\end{array}\right.
\end{equation}
Then, for each $(\lambda,x)\in\Lambda\times B(x_0,\rho)$ the derivative
$\Phi\sp{\,\prime}_x(\lambda,x)$ of $\Phi(\lambda,\cdot)$ at $x$ is given
by
\begin{equation}\label{de}\hskip-2mm
  \Big(\Phi\sp{\,\prime}_x(\lambda,x)\,z\Big)(s)\,{=}\,z(T)\,{+}\hskip-1mm
  \int_0^s\hskip-1mm D[F\sp{\,\prime}_{x}(\lambda,x(\tau),t)\,z(\tau)]
  \mbox{\ for\ } z\in G\OT \mbox{\ and\ } s\in\OT.
\end{equation}
\end{proposition}
\begin{proof}
First, recall that  $F\sp{\,\prime}_x(\lambda,\cdot,\cdot)
\in\mathcal{F}(\Omega\times\OT,\wt{h},\wt{\omega};\mathcal{L}(\R^n))$
means that
\begin{align}\label{C11}
 &\left\{\begin{array}{l}
   \|F\sp{\,\prime}_{x}(\lambda,x,t)-F\sp{\,\prime}_{x}(\lambda,x,s)\|_{n\times n}
   \le|\wt{h}(t)-\wt{h}(s)|
\\[1mm]
   \hskip35mm\mbox{for\ } \lambda\in\Lambda, x\in\Omega\mbox{\ and\ } t,s\in\OT.
  \end{array}\right.
\\\noalign{\noindent\mbox{and}\vskip-7mm}\nonumber
\\\label{C12}
 &\left\{\begin{array}{l}
    \|F\sp{\,\prime}_{x}(\lambda,x,t)-F\sp{\,\prime}_{x}(\lambda,x,s)
    -F\sp{\,\prime}_{x}(\lambda,y,t)+F\sp{\,\prime}_{x}(\lambda,y,s)\|_{n\times n}
 \\[1mm]\hskip10mm
  \le\wt{\omega}(\|x{-}y\|_n)\,|\wt{h}(t)\,{-}\,\wt{h}(s)|
 \\[1mm]\hskip35mm
  \mbox{for\ }\lambda\in\Lambda,\, x,y\in\Omega\mbox{\ and\ } t,s\in\OT.
  \end{array}\right.
\end{align}
By Proposition \ref{equiv}, $\Phi$ maps $B(x_0,\rho)$ into $G\OT$ for any
$\lambda\in\Lambda.$ Let $x\in B(x_0,\rho)$ and $\lambda\in\Lambda$ be given.
By \eqref{B2}, $x(t)\in\Omega$ for all $t\in\OT.$ Consider the operator
function \ $\Psi$ \ defined by
\[
  \Big(\Psi(\lambda,x)\,z\Big)(r)
  =z(T)+\int_0^r D[F\sp{\,\prime}_{x}(\lambda,x(\tau),t)\,z(\tau)]
  \mbox{ \ for\ } z\in G\OT \mbox{\ and\ } r\in\OT.
\]
Obviously, $\Psi(\lambda,x):G\OT\to G\OT$ \ is linear and bounded. Indeed,
by Lemma \ref{Ant} and \eqref{C11} we have
\begin{align*}
  \|\Psi(\lambda,x)\,z\|_{\infty}
  &=\sup_{r\in\OT}\|\Big(\Psi(\lambda,x)\,z\Big)(r)\|_n
 \\
  &=\sup_{r\in\OT}\left\|z(T)
      +\int_0^r D[F\sp{\,\prime}_{x}(\lambda,x(\tau),t)\,z(\tau)]\right\|_n
 \\
  &\le\|z(T)\|_n+\sup_{r\in\OT}\int_0^r\|z(\tau)\|_n\,d \wt{h}(\tau)
   \le[1+(\wt{h}(T)-\wt{h}(0))]\,\|z\|_{\infty}
\end{align*}
for each $z\,{\in}\,G\OT.$

\medskip

We want to show that
\begin{equation}\label{lim}
   \lim_{\vartheta\to 0+}
   \left\|\dfrac{\Phi(\lambda,x+\vartheta\,z)-\Phi(\lambda,x)}{\vartheta}
          -\Psi(\lambda,x)\,z\right\|_\infty=0
   \quad\mbox{for all\ } z\in G\OT.
\end{equation}
To this aim, let $z\in G\OT$ \ be given. Then, for every $r\in\OT$ and
$\vartheta\in (0,1)$ sufficiently small we have $x+\vartheta\,z\in B(x_0,\rho)$
and
\[
 \frac{\Phi(\lambda,x+\vartheta\,z)(r)-\Phi(\lambda,x)(r)}
      {\vartheta}-(\Psi(\lambda,x)\,z)(r)
 =\int_0^r DU(\tau,t),
\]
where
\begin{equation}\label{U}\hskip-11mm
\left\{\begin{array}{l}\displaystyle
  U(\tau,t)=\dfrac{F(\lambda,x(\tau)+\vartheta\,z(\tau),t)-F(\lambda,x(\tau),t)}
                  {\vartheta}
               -F\sp{\,\prime}_{x}(\lambda,x(\tau),t)\,z(\tau)
 \\[3mm]
  \mbox{for \ }\tau,t\in\OT.
\end{array}\right.
\end{equation}
Notice, that due to convexity of $B(x_0,\rho),$ the functions
$\alpha\,(x+\vartheta\,z)+(1-\alpha)\,x$ belong to $B(x_0,\rho)$
for each $\alpha\in[0,1].$ In particular,
$\alpha\,(x(\tau)+\vartheta\,z(\tau))+(1-\alpha)\,x(\tau)\in\Omega$
for all $\tau\in\OT$ and $\alpha\in[0,1].$
Thus, we can use the Mean Value Theorem for vector-valued functions
(see e.g. \cite[Lemma 8.11]{Peter}) to verify that the relations
\begin{multline*}
 F(\lambda,x(\tau)+\vartheta\,z(\tau),t)-F(\lambda,x(\tau),t)
\\
 =\left[\int_0^1
    F\sp{\,\prime}_{x}
     (\lambda,\alpha\,(x(\tau)+\vartheta\,z(\tau)){+}(1-\alpha)\,x(\tau),t)\,d \alpha
  \right]\vartheta\,z(\tau)
\end{multline*}
are true for arbitrary $t,\tau\in\OT.$ Hence, we can rearrange the difference
$U(\tau,t)-U(\tau,s)$ as follows
\begin{equation}\label{4.6}
\hskip-6mm
\left\{\begin{array}{l}\displaystyle
 U(\tau,t)-U(\tau,s)
 =\Bigg[
   \int_0^1\Big[F\sp{\,\prime}_{x}(\lambda,\alpha\,(x(\tau){+}\vartheta\,z(\tau))+(1{-}\alpha)\,x(\tau),t)
 \\[5mm]\hskip42mm
   -F\sp{\,\prime}_{x}(\lambda,\alpha\,(x(\tau){+}\vartheta\,z(\tau))+(1{-}\alpha)\,x(\tau),s)\big]d\alpha
 \\[3mm]\hskip42mm\displaystyle
   -\int_0^1\big[F\sp{\,\prime}_{x}(\lambda,x(\tau),t)-F\sp{\,\prime}_{x}(\lambda,x(\tau),s)\big]\,d\alpha\Bigg]\,z(\tau)
 \\[3mm]
  \mbox{for \ } t,s,\tau\in\OT.
\end{array}\right.\hskip10mm
\end{equation}
Furthermore, using \eqref{C12} we obtain
\begin{equation}\label{4.7}
\hskip-12mm\left\{\begin{array}{l}\displaystyle
  \Big\|F\sp{\,\prime}_{x}
          (\lambda,\alpha\,(x(\tau){+}\vartheta\,z(\tau)){+}(1{-}\alpha)\,x(\tau),t)
\\[2mm]\hskip8mm\displaystyle
       -F\sp{\,\prime}_{x}
          (\lambda,\alpha\,(x(\tau){+}\vartheta\,z(\tau)){+}(1{-}\alpha)\,x(\tau),s)
\\[2mm]\hskip8mm\displaystyle
       -F\sp{\,\prime}_{x}
         (\lambda,x(\tau),t)+F\sp{\,\prime}_{x}(\lambda,x(\tau),s)\Big\|_{n\times n}
   \le \wt{\omega}(\vartheta\,\|z\|_\infty)\,|\wt{h}(t)-\wt{h}(s)|
\\[2mm]
   \mbox{for \ } \vartheta\in[0,1] \mbox{\ and\ } t,s,\tau\in\OT.
\end{array}\right.
\end{equation}
Inserting \eqref{4.7} into \eqref{4.6}, we verify that the inequality
\[
  \|U(\tau,t)-U(\tau,s)\|_n
  \le \wt{\omega}(\vartheta\,\|z\|_\infty)\,|\wt{h}(t)-\wt{h}(s)|\,\|z\|_\infty
\]
holds for all $t,s,\tau\in[0,t].$ Finally, making use of Lemma \ref{Ant} we achieve
the inequality
\[
   \sup_{r\in\OT}\Big\|\int_0^r DU(\tau,t)\Big\|_n
   \le\int_0^T \wt{\omega}(\vartheta\,\|z\|_\infty)\,d \wt{h}\,\|z\|_\infty
   =\wt{\omega}(\vartheta\,\|z\|_\infty)\,[\wt{h}(T)-\wt{h}(0)]\,\|z\|_\infty.
\]

This, together with \eqref{U}, implies the relations
\begin{align*}
& 0\le\lim_{\vartheta\to 0+}
 \left\|
    \dfrac{\Phi(\lambda,x+\vartheta\,z)-\Phi(\lambda,x)}{\vartheta}-\Psi(\lambda,x)\,z
 \right\|_\infty
\\[2mm]&\hskip25mm
 \le\lim_{\vartheta\to 0+}\wt{\omega}(\vartheta\,\|z\|_\infty)
                                \,[\wt{h}(T)-\wt{h}(0)]\,\|z\|_\infty=0,
\end{align*}
i.e. the desired relation \eqref{lim} is true. This completes the proof.
\end{proof}

Next two propositions show that when we include the conditions \eqref{B3} and/or
similar condition \eqref{C3} on the derivative of $F$, we reach the continuity
of $\Phi$ and of its derivative on $\Lambda\times B(x_0,\rho).$

\begin{proposition}\label{ucG}
Assume that \eqref{B1}, \eqref{B2}, \eqref{B3} are satisfied and let
$\Phi$ be given by \eqref{operador}. Then $\Phi$ is continuous on
$\Lambda\times B(x_0,\rho).$
\end{proposition}
\begin{proof}
Let  $(\lambda_1,x),(\lambda_2,y)\in\Lambda\times B(x_0,\rho)$ and $s\in\OT$
be given. Obviously, we have
\begin{equation}\label{aa}
   [\Phi(\lambda_1,x)-\Phi(\lambda_2,y)](s)=
   x(T)-y(T)+\int_0^s DF(\lambda_1,x(\tau),t)-F(\lambda_2,y(\tau),t),
\end{equation}
where
\begin{align*}
  &\int_0^s D[F(\lambda_1,x(\tau),t)-F(\lambda_2,y(\tau),t)]
\\
  &\quad=\int_0^s D[F(\lambda_1,x(\tau),t)-F(\lambda_1,y(\tau),t)]
   +\int_0^s D[F(\lambda_1,y(\tau),t)-F(\lambda_2,y(\tau),t)].
\end{align*}
Furthermore,
\begin{equation}\label{bb}
   \left\|\int_0^s D[F(\lambda_1,x(\tau),t)-F(\lambda_1,y(\tau),t)]\right\|_n
   \le\omega(\|x-y\|_\infty)\,[h(T)-h(0)]
\end{equation}
due to Lemma \ref{JA}.

Now, let $\eps>0$ be given and let $\delta\in(0,\eps)$ be such that \eqref{B3}
is true. Then Lemma \ref{Ant} implies that also the relation
\[
  \left\|\int_0^s D[F(\lambda_1,y(\tau),t)-F(\lambda_2,y(\tau),t)]\right\|_n
  <\eps\,[\gamma(T)-\gamma(0)]
\]
holds whenever $|\lambda_1-\lambda_2|<\delta.$
To summarize, inserting the last relation together with \eqref{bb} into \eqref{aa}
we obtain
\begin{align*}
  \|\Phi(\lambda_1,x)-\Phi(\lambda_2,y)\|_{\infty}
  &\le\|x-y\|_{\infty}+\omega(\|x-y\|_{\infty})\,[h(T)-h(0)]
     +\eps\,[\gamma(T)-\gamma(0)]
\\[1mm]
  &<\eps\,(1+[h(T)-h(0)]+[\gamma(T)-\gamma(0)])
\end{align*}
whenever $\|x-y\|_\infty$ is sufficiently small. In other words, the operator
function $\Phi$ is continuous on $\Lambda\times B(x_0,\rho).$
\end{proof}

\begin{proposition}\label{cont}
Let the assumptions of Proposition \ref{derivative} be satisfied and let
\begin{equation}\label{C3}\hskip-2mm
\left\{\begin{array}{l}
\mbox{there is a~nondecreasing function $\wt{\gamma}:\OT\to\R$ such that for}
\\[1mm]
\mbox{any $\eps>0$ there is a $\delta>0$ such that}
\\[2mm]\displaystyle
   \|F\sp{\,\prime}_{x}(\lambda_1,x,t)-F\sp{\,\prime}_{x}(\lambda_2,x,t)
          -F\sp{\,\prime}_{x}(\lambda_1,x,s)+F\sp{\,\prime}_{x}(\lambda_2,x,s)\|_{n\times n}
   <\eps\,|\wt{\gamma}(t)-\wt{\gamma}(s)|
\\[1mm]\hskip5mm
   \mbox{for\ } x\in\Omega,\,t,s\in\OT\mbox{\ and\ }
   \lambda_1,\lambda_2\,{\in}\,\Lambda
   \mbox{\ such that\ } |\lambda_1-\lambda_2|<\delta.
\end{array}\right.
\end{equation}
Then the operator function $\Phi\sp{\,\prime}_{x}:\Lambda\times
  B(x_0,\rho)\to\mathcal{L}(G\OT)$ is continuous.
\end{proposition}
\noindent
{\it Proof} \ is quite analogous to that of Proposition \ref{ucG},
only instead of $\Phi(\lambda,x)$ and $F(\lambda,x(\tau),t)$ we
should respectively deal with $\Phi\sp{\,\prime}_{x}(\lambda,x)\,z$
and $F\sp{\,\prime}_{x}(\lambda,x(\tau,t)\,z(\tau),$ where
$z\in G\OT.$\ $\square$

\begin{theorem}\label{der}
Let \eqref{B3} and all the assumptions of Proposition \ref{cont}
be satisfied, let $\lambda_0\,{\in}\Lambda$ be given and let
$Id-\Phi\sp{\,\prime}_x(\lambda_0,x_0)$ be an isomorphism of
\ $G\OT$ onto $G\OT.$ Then there is $\delta>0$ such that $(x,\lambda)$
is not a~bifurcation point of the equation $\Phi(\lambda,x)=x$
whenever $\|x-x_0\|_\infty+|\lambda-\lambda_0|<\delta.$
\end{theorem}
\begin{proof}
First, recall that, according to Propositions \ref{derivative},
\ref{ucG} and \ref{cont}, the operator function $\Phi(\lambda,\cdot)$
is continuous together with its derivative
$\Phi\sp{\,\prime}_x(\lambda,x)\in\mathcal{L}(G\OT)$ on
$\Lambda\times B(x_0,\rho).$ Further, by \eqref{B2} we have
\begin{equation}\label{*}
  x_0=\Phi(\lambda,x_0) \quad\mbox{for all \ } \lambda\in\Lambda.
\end{equation}

Let $Id-\Phi\sp{\,\prime}_x(\lambda_0,x_0)$ be an isomorphism of
$G\OT$ onto $G\OT.$ By the Implicit Function Theorem (see e.g.
\cite[Theorem 4.2.1]{DM}) this means that there exist neighborhoods
$\mathcal{V}\subset\Lambda$ of $\lambda_0$ and $\mathcal{W}\subset B(x_0,\rho)$
of $x_0$ such that for any $\lambda\in\mathcal{V}$ there is a~unique
$x\in\mathcal{W}$ such that \ $x=\Phi(\lambda,x).$ However, this together
with \eqref{*} implies that $x=x_0$ has to be the only function satisfying
the relations
\[
   x=\Phi(\lambda,x)
   \quad\mbox{for any \ } \lambda\in\mathcal{V}\subset\Lambda.
\]
Hence, according to Definition \ref{bif}, $(x_0,\lambda_0)$ is not
a~bifurcation point of the equation $x\,{=}\,\Phi(\lambda,x).$
The proof will be completed by using Corollary \ref{cor-def}.
\end{proof}

Next assertion provides a related Fredholm Alternative type result.

\begin{theorem}\label{Alt}
Let the assumptions of Proposition \ref{derivative} be satisfied and
$x_0\in G\OT$ is given. Then, either
\begin{itemize}
\item[{\rm(i)}] the equation\vskip-6mm
\begin{equation*}\label{200}
   z(s)-z(T)-\int_0^s D[F\sp{\,\prime}_{x}(\lambda_0,x_0,t)\,z(\tau)]=q(s)
   \quad\mbox{for\ } s\in\OT
\end{equation*}\vskip-3mm
has a~unique solution in $G\OT$ for every $q\in G\OT;$\vskip-2mm
\hskip-8mm or
\item[{\rm(ii)}] the corresponding homogeneous equation\vskip-6mm
\begin{equation*}\label{3.12}
  z(s)-z(T)-\int_0^s D[F\sp{\,\prime}_{x}(\lambda_0,x_0,t)\,z(\tau)]=0
  \quad\mbox{for\ } s\in\OT
\end{equation*}\vskip-2mm
has at least one nontrivial solution in $G\OT.$
\end{itemize}
\end{theorem}
\begin{proof}
Let $\Phi$ be defined by \eqref{operador} and let $\rho>0$ be given by
\eqref{B2}. Let $\lambda\in\Lambda$ and $x\in B(x_0,\rho)$ be given.
By Proposition \ref{derivative}, we have
\[
  \big(\Phi\sp{\,\prime}_x(\lambda,x)\,z\big)(r)
  =z(0)+\int_0^r D[F\sp{\,\prime}_{x}(\lambda,x,t)\,z(\tau)]
  \quad\mbox{for \ } z\in G\OT \mbox{\ and\ } r\in \OT.
\]
We assert that $\Phi\sp{\,\prime}_x(\lambda,x)$ is a~compact operator
on $G\OT.$ Indeed,
it is linear and bounded as it was shown in the beginning of the proof
of Proposition \ref{derivative}. Hence, it remains to show that it maps
bounded subsets of $G\OT$ onto relatively compact subsets of $G\OT.$

Let $M\subset G\OT$ be bounded and let $c>0$ be such that $\|z\|_{\infty}\le c$ for all
$z\in M.$  Making use of \eqref{C1} and Lemma \ref{Ant},
we get
\begin{align*}
   &\|(\Phi\sp{\,\prime}_x(\lambda,x)\,z)(r')-(\Phi\sp{\,\prime}_x(\lambda,x)\,z)(r)\|_\infty
  \\[2mm]&\quad
  =\left\|\int_{r}^{r'} D[F\sp{\,\prime}_{x}(\lambda,x,t)\,z(\tau)]\right\|_n
  \le\int_{\min\{r,r'\}}^{\max\{r,r'\}}\|z(\tau)\|_n\,d\wt{h}(\tau)
  \le c\,|\wt{h}(r')-\wt{h}(r)|
\end{align*}
for all $r,r'\in\OT$ and $z\in M.$ By \cite[Theorem 2.17]{dana}
(cf. also \cite[Corollary 4.3.8]{MST}), the set $\{\Phi\sp{\,\prime}_x(\lambda,x)\,z):z\in M\}$
is relatively compact. This proves our claim.

Therefore, using the Fredholm Alternative for Banach spaces (see e.g.
\cite[Theorem 4.12]{Sche}), we have that either the range
$\mathcal{R}(Id-\Phi\sp{\,\prime}_x(\lambda,x))$ of the operator $Id-\Phi\sp{\,\prime}_x(\lambda,x)$
is the whole $G\OT$ and its null space $\mathcal{N}(Id-\Phi\sp{\,\prime}_x(\lambda,x))=\{0\}$
or $\mathcal{R}(Id-\Phi\sp{\,\prime}_x(\lambda,x))\ne G\OT$ and
$\mathcal{N}(Id-\Phi\sp{\,\prime}_x(\lambda,x))\ne\{0\}.$ This completes the proof.
\end{proof}

Now, we can reformulate conditions necessary for $(\lambda_0,x_0)$ to be
a~bifurcation point of the equation $\Phi(\lambda,x)=x$ as follows:

\begin{theorem}\label{NC}
Suppose that the assumptions of Theorem \ref{der} are satisfied and let
$\lambda_0\,{\in}\,\Lambda$ and $x_0\,{\in}\,B(x_0,\rho)$ be given.
Then, $(x_0,\lambda_0)$ is a~bifurcation point of the equation
$\Phi(\lambda,x)\,{=}\,x$ only if there exists $q\in G\OT$ such that
the equation
\begin{equation}\label{11}
   z(s)-z(T)-\int_0^s D[F\sp{\,\prime}_{x}(\lambda_0,x_0,t)\,z(\tau)]=q(s)
   \quad\mbox{for \ } s\in\OT
\end{equation}
has no solution in $G\OT$ and the corresponding homogeneous equation
\[
   z(s)-z(T)-\int_0^s D[F\sp{\,\prime}_{x}(\lambda_0,x_0,t)z(\tau)]=0
   \quad\mbox{for \ } s\in\OT
\]
possesses at least one nontrivial solution in $G\OT.$
\end{theorem}
\begin{proof}
Suppose $(x_0,\lambda_0)$ is a~bifurcation point of the equation
$\Phi(\lambda,x)\,{=}\,x.$ Then, by Theorem \ref{der}, the operator
$Id-\Phi\sp{\,\prime}_x(\lambda_0,x_0):G\OT\to G\OT$ can not be an
isomorphism. Therefore, using Theorem \ref{der} and Fredholm type
Alternative~\ref{Alt}, we conclude that
$\mathcal{R}(Id-\Phi\sp{\,\prime}_x(\lambda_0,x_0))\ne G\OT$ and
$\mathcal{N}(Id-\Phi\sp{\,\prime}_x(\lambda_0,x_0))\ne\{0\}.$ Our
statement follows immediately.
\end{proof}

\begin{remark}
Notice that \eqref{11} is the periodic problem for a nonhomogeneous generalized
linear differential equation.
\end{remark}

\section{Measure Differential Equations}

Main topic of this paper are measure differential equations of the form
\begin{equation}\label{MDE}
    D x=f(\lambda,x,t)+g(x,t)\,.\,D h,
\end{equation}
where
\begin{equation}\label{A1}
\left\{\begin{array}{l}
  \Omega\subset\R^n\mbox{\ and \ } \Lambda\subset\R
  \mbox{ \ are open sets};
  \\[2mm]
  f:\Lambda\times\Omega\times\OT\to\R^n,\, g:\Omega\times\OT\to\R^n;
  \\[2mm]
  h:(-\infty,T]\to\R \mbox{\ is left-continuous and has a bounded}
  \\[1mm]
  \mbox{variation on\ } \OT
  \mbox{\ and\ } h(t){=}\,h(0) \mbox{ \ for\ } t<0;
  \\[2mm]
  x:\OT\to\R^n;\,
  D x \mbox{\ is the (Schwartz) distributional derivative of\ } x;
  \\[1mm]
  D h \mbox{\ is the (Schwartz) distributional derivative of\ } h.
 \end{array}\right.
\end{equation}

It is well known that such kind of differential equations, usually
called distributional or measure, encompass many types of equations
such as ordinary differential equations, impulsive differential
equations, dynamic equations on time scales and others.

\begin{remark}\label{R1}{\bf (Distributions.)}
By distributions we understand linear continuous func\-tio\-nals on
the topological vector space $\mathcal{D}$ of functions
$\varphi:\R\to\R$ possessing for any $j\in N\cup\{0\}$ a derivative
$\varphi^{(j)}$ of the order $j$ which is continuous on $\R$ and such
that $\varphi^{(j)}(t)=0$ if $t\notin (0,T)$. The space $\mathcal{D}$
is endowed with the topology in which the sequence
$\varphi_k\in\mathcal{D}$ tends to $\varphi_0\in\mathcal{D}$ in
$\mathcal{D}$ if and only if
\[
   \lim_k\|\varphi_k^{(j)}-\varphi_0^{(j)}\|_\infty=0
   \mbox{\ for all non negative integers\ }j.
\]
Similarly, $n$-vector distributions are linear continuous $n$-vector
functionals on the $n$-th cartesian power $\mathcal{D}^n$ of
$\mathcal{D}.$ The space of $n$-vector distributions on $\OT$ (the
dual space to $\mathcal{D}^n$) is denoted by $\mathcal{D}^{n*}.$
Instead of $\mathcal{D}^{1*}$ we write $\mathcal{D}^*.$ Given
a distribution $f\in\mathcal{D}^{n*}$ and a (test) function
$\varphi\in\mathcal{D}^n,$ the value of the functional $f$ on $\varphi$
is denoted by $<f,\varphi>$. Of course, reasonable real valued point
functions are naturally included between distributions. For example,
for a given $f$ Lebesgue integrable on $\OT$ ($f\in L^1\OT$), the
relation
\[
    <f,\varphi>=\int_0^T f(t)\,\varphi(t)\,d t
    \quad\mbox{for \ }\varphi\in\mathcal{D}^n,
\]
(where $f(t)\,\varphi(t)$ stands for the scalar product of $f(t)\in\R^n$
and $\varphi(t)\in\R^n$) defines the $n$-vector distribution on $\OT$
which will be denoted by the same symbol $f.$ As a result, the zero
distribution $0\in\mathcal{D}^{n*}$ on $[0,T]$ can be identified with
an arbitrary measurable function vanishing a.e. on $[0,T]$. Obviously,
if $f\in G\OT$ is left-continuous on $(0,T],$ then
$f=0\in\mathcal{D}^{*n}$ only if $f(t)=0$ for all $t\in\OT.$

Given two distributions $f,g\in\mathcal{D}^{n*},$ $f=g$ means that
$f-g=0\in\mathcal{D}^{n*}.$ Whenever a~relation of the form $f=g$ for
distributions and/or functions $f$ and $g$  occurs in the following
text, it is understood as the equality in the above sense. Given an
arbitrary $f\in\mathcal{D}^{n*},$ the symbol $Df$ denotes its
distributional derivative, i.e.
\[
   <Df,\varphi>=-<f,\varphi'>
   \quad\mbox{for \ } \varphi\in\mathcal{D}^n.
\]
For absolutely continuous functions their distributional derivatives
coincide with their classical derivatives, of course. It is well-known,
cf. \cite[Section 3]{Ha}, that if $f\in\mathcal{D}^*,$ then $Df=0$ if
and only if $f$ is Lebesgue integrable on $\OT$ and there is
a $c_0\in\R$ such that $f(t)=c_0$ a.e. on $\OT.$

For more details on the theory of distributions, see e.g. \cite{FJ},
\cite{Ka}, \cite[Chapter 6]{Ru}, \cite[Section 8.4]{MST}, \cite{T}.
\end{remark}

\begin{definition}\label{D}
By a solution of \eqref{MDE} we understand a couple
$(x,\lambda)\in BV\OT\times\Lambda$ such that $x$ is left-continuous on
$(0,T],$ $x(t)\,{\in}\,\Omega$ for $t\,{\in}\,\OT,$ the distributional
product $\wt{g}_x\,{.}\,D h$ of the function
\[
   \wt{g}_{x}:t\in\OT\to g(x(t),t)\,{\in}\,\R^n
\]
with the distributional derivative $Du$ of $u$ has a~sense and the
equality \eqref{MDE} is satisfied in the distributional sense, i.e.
\[
   <Dx,\varphi>=<\wt{f}_{\lambda,x},\varphi>+<\wt{g}_x\,{.}\,D h,\varphi>
   \quad\mbox{for all \ } \varphi\in\mathcal{D}^n,
\]
where $\wt{f}_{\lambda,x}:t\in\OT\to f(\lambda,x(t),t)\in\R^n.$
\end{definition}

\begin{remark}\label{R2}
According to Definition \ref{D}, to investigate differential equations
like \eqref{MDE}, one should reasonably specify how to understand to
the distributional product $\wt{g}_x\,{.}\,D h,$ symbolically written as
$g(x,t)\,{.}\,D h,$ on the right-hand side of equation \eqref{MDE}. It is known
that in the Schwartz setting it is not possible to define a product
of an arbitrary couple of distributions. In text-books one can find
the trivial example when $f\in\mathcal{D}^*$ and $g:\OT\to\R$ is
infinitely differentiable on $\OT$ and its support is contained in
the open interval $(0,T).$ The product $f.g$ of $f$ and $g$ is in such
a case defined as
\[
   \langle f\,g,\varphi\rangle=\langle f,g\varphi\rangle
   \quad\mbox{for all \ } \varphi\in\mathcal{D}.
\]
Furthermore, if $f,g\in L^1\OT$ are such that $f\,g\in L^1\OT,$ their
distributional product is defined as
\[
   \langle f\,g,\varphi\rangle=\int_0^T f(t)\,g(t)\,\varphi(t)\,d t
   \quad\mbox{for \ } \varphi\in\mathcal{D}^n.
\]
Thus, in this case the distributional product actually coincide with
the usual product of point functions. However, in equation \eqref{MDE}
we have a product of a $n$-vector valued function with the distributional
derivative of a scalar function which is evidently not covered by the
above definitions. The definition of a product of measures and regulated
functions given by Lig\c{e}za in \cite{Li} on the basis of the sequential
approach is unfortunately not suitable for our purposes. As will be seen
below, a good tool in the context of measure differential systems is
provided by the Kurzweil-Stieltjes integral. The following definition has
been introduced in \cite{T}, cf. also \cite[Section 8.4]{MST}.
\end{remark}

\begin{definition}\label{DP}
If $g:\OT\to\R^n$ and $h:\OT\to\R$ are functions defined on $\OT$ and
such that there exists the Kurzweil-Stieltjes integral
$\int_0^T g\,d h,$ then the product of $g$ and $D h$ is the
distributional derivative of the indefinite integral
$H(t):=\int_0^t g\,d h,$ i.e. $g\,{.}\,D h=D H.$
\end{definition}

\begin{remark}\label{R3}
Note that in Definition \ref{DP}, the product $g\,{.}\,D h$ is an $n$-vector
distribution.

Furthermore, it is worth mentioning that the multiplication operation
given by Definition \ref{DP} is associative, distributive and
multiplication by zero element gives zero element. On the other hand,
we should have in mind that (cf. \cite[Remark 4.1]{T} and
\cite[Theorem 6.4.2]{MST}) the expected formula
\[
   D(f\,{.}\,g)=D f\,{.}\,g+f\,{.}\,D g
\]
for the differentiation of the product $f.g$ is not true, in general.
More precisely, using the modified integration-by-parts formula from
\cite[Theorem 6.2]{AST} one can verify that the following relation
holds if $f$ and $g$ are regulated and at least one of them has a bounded
variation
\begin{align*}
   &D(f.g)=D f.g+f\,{.}\,D g + Df.\Delta^+\wt{g}-\Delta^-\wt{f}\,{.}\,Dg,
  \\\noalign{\noindent\mbox{where}}
   &\Delta^+\wt{g}(t)
     =\left\{\begin{array}{ll}
               \Delta^+g(t) &\mbox{if\ } t<T,
              \\
               0            &\mbox{if\ } t=T
       \end{array}\right.
    \quad\mbox{and}\quad
    \Delta^-\wt{f}(t)
     =\left\{\begin{array}{ll}
               0            &\mbox{if\ } t=0,
              \\
               \Delta^-f(t) &\mbox{if\ } t>0.
       \end{array}\right.
\end{align*}
\end{remark}

Together with \eqref{MDE} we will consider the Stieltjes integral equation
\begin{equation}\label{SIE}
   x(t)=x(0)+\int_0^t f(\lambda,x(s),s)\,d s+\int_0^t g(x(s),s)\,d h(s)
   \quad\mbox{for \ } t\in\OT,
\end{equation}
where the integrals stand for the Kurzweil-Stieltjes ones
\footnote{Recall that the Kurzweil-Stieltjes integral with the identity
integrator becomes the Henstock-Kurzweil one.}.

By a solution we understand any function $x:\OT\to\R^n$ such that $x(t)\in\Omega$
for $t\in\OT$ and the equality \eqref{SIE} is true on $\OT.$

\begin{remark}
In the literature one often meets instead of the integral version \eqref{SIE}
of \eqref{GDEper} the integral equation
\begin{equation}\label{Leb}
   x(t)=x(0)+\int_0^t f(\lambda,x(s),s)\,d s+\int_{[0,t)}g(x(s),s)\,d\mu_u,
\end{equation}
where the former integral is the Lebesgue one and the latter is the Lebesgue-Stieltjes
integral. However, it is known, cf. \cite[Theorem 6.12.3]{MST}, that if the
Lebesgue-Stieltjes integral $(LS)\int_{[0,T)}g\,d\mu_u$ exists, then the
Kurzweil-Stieltjes integral $\int_0^T g\,du$ exists as well and
\footnote{Recall that $u$ is left-continuous on $(0,T]$ and $u(0-)=u(0).$}
\[
  \int_0^T g\,du=(LS)\int_{[0,T)} g\,d\mu_u.
\]
Therefore, equation \eqref{Leb} is a special case of \eqref{SIE}.
\end{remark}


\begin{proposition}\label{P1}
Assume that conditions \eqref{A1},
\begin{equation}\label{A2i}
\left\{\begin{array}{l}
  f(\lambda,\cdot,t) \mbox{\ is continuous on\ } \Omega
  \mbox{\ for all\ } t\in\OT \mbox{\ and\ } \lambda\in\Lambda;
  \\[2mm]
  f(\lambda,x,\cdot) \mbox{\ is Lebesgue measurable on\ } \OT
  \mbox{\ for all\ } (\lambda,x)\,{\in}\,\Lambda\times\Omega;
  \\[2mm]
  \mbox{there is a function\ }m\,{:}\,\OT\,{\to}\,[0,\infty)
  \mbox{\ Lebesgue integrable}
  \\
  \mbox{on\ } \OT \mbox{\ and such that}
  \\[2mm]\displaystyle\hskip10mm
    \|f(\lambda,x,t)\|_n\le m(t)
    \mbox{ \ for\ } (\lambda,x,t)\in\Lambda\times\Omega\times\OT
\end{array}\right.
\end{equation}
and
\begin{equation}\label{A2ii}\hskip-18mm
\left\{\begin{array}{l}
  g(\cdot,t) \mbox{\ is continuous on\ } \Omega \mbox{\ for all\ } t\in\OT
  \mbox{\ and there is}
  \\
   \mbox{a function\ }m_u{:}\,\OT\,{\to}\,[0,\infty) \mbox{\ such that}
  \\[1mm]\displaystyle\hskip5mm
    \|g(x,t)\|_n\le m_u(t)
    \mbox{ \ and \ } \int_0^T m_u(t)\,d[\var_0^t u]\,{<}\,\infty
  \\[2mm]
    \mbox{for \ } (x,t)\in\Omega\times\OT.
\end{array}\right.
\end{equation}
are satisfied.

Then any solution $x$ of \eqref{SIE} on $\OT$ is left-continuous on $(0,T]$ and has
a bounded variation on $\OT.$
\end{proposition}
\begin{proof}
Let $x$ be a solution of \eqref{SIE}. Then $x(t)\in\Omega$ for all $t\in\OT$ and both
integrals on the right hand side of \eqref{SIE} have a sense for all $t\in\OT.$
Due to the condition \eqref{D1i}, the integral $\int_0^T f(\lambda,x(s),s)\,d s$
exists as the Lebesgue one and as a result the corresponding indefinite integral
is absolutely continuous on $\OT.$

Furthermore, denote
\[
   H(t):=\int_0^t g(x(s),s)\,d h(s) \quad\mbox{for \ } t\in\OT.
\]
By \cite[Corollary 6.5.5]{MST}, $H$ is left-continuous on $(0,T].$ Furthermore,
due to \eqref{A2ii} and \cite[Theorem 6.7.4]{MST}, the integral
\ $\int_c^d\|g(x(s),s)\|_n\,d[\var_0^s\,h]$ \ exists for each $[c,d]\subset\OT.$
Consequently, for an arbitrary division $\{\alpha_0,\alpha_1,\dots,\alpha_m\}$
of $\OT$ we get
\begin{align*}
	&\sum_{j=1}^m\|H(\alpha_j)-H(\alpha_{j-1})\|_n
	\le\sum_{j=1}^m\int_{\alpha_{j-1}}^{\alpha_j}\|g(x(s),s)\|_n\, d[\var_0^s\,h]
   \\
    &\quad\le\int_0^T m_u(s)\,d[\var_0^s\,h]<\infty,
\end{align*}
i.e. $H$ has a bounded variation on $\OT.$ This completes the proof.
\end{proof}

\begin{theorem}\label{T1}
Let conditions \eqref{A1}, \eqref{A2i} and \eqref{A2ii} be satisfied.
Then $x\,{\in}\,G\OT$ is a solution of \eqref{MDE} on $\OT$ if and only if
it is a solution to \eqref{SIE}.
\end{theorem}
\begin{proof}
If $x$ is a solution to \eqref{SIE}, then it is a solution to \eqref{MDE} on $\OT$
thanks to Proposition \ref{P1} and Definition \ref{DP}.

On the other hand, let $x$ be a solution of \eqref{MDE}. By Definition \ref{D},
$x$ is left-continuous on $(0,T],$ has a bounded variation on $\OT$ and
$x(t)\in\Omega$ for all $t\in\OT.$ Furthermore, by definition \ref{DP},
\begin{gather*}
   D\big(x-F_{\lambda}(x)\big)=0\in\mathcal{D}^{n*},
\\\noalign{\noindent\mbox{where}}
   F_{\lambda}(x):t\in\OT\to
   \int_0^t f(\lambda,x(s),s)\,d s+\int_0^t g(x(s),s)\,d h(s)\in\R^n
   \mbox{\ for\ } \lambda\in\Lambda.
\end{gather*}
By the proof of Proposition \ref{P1} $F_{\lambda}(x)$ has a bounded variation
on $\OT$ and is left-continuous on $(0,T]$ for all $\lambda\in\Lambda.$
By \cite[Section 3]{Ha} this means that there is $c\in\R^n$ such that
$x(t)-F_{\lambda}(x)(t)=c$ for all $\lambda\in\Lambda$ and $t\in\OT.$ As
a result, $c=x(0)$ and $x$ is a solution to \eqref{SIE}.
\end{proof}

Let us consider the functions $F_1,$ $F_2$ and $F$ given for
$(\lambda,x,t)\in\Lambda\times\Omega\times\OT$ by the relations
\begin{equation}\label{F}\hskip-7mm
\left\{\begin{array}{l}\displaystyle
   F_1(\lambda,x,t)=\int_0^t f(\lambda,x,s)\,d s,\quad
   F_2(x,t)=\int_0^t g(x,s)\,d h(s),
  \\[4mm]\displaystyle
   F(\lambda,x,t)=F_1(\lambda,x,t)+F_2(x,t)
\end{array}\right.
\end{equation}
whenever the integrals on the right-hand sides have a sense.

\medskip

Next two assertions follows immediately from \cite[Proposition 4.7]{Schw} and
\cite[Proposition 4.8]{Schw}, respectively.
\begin{proposition}\label{Schw}
Let the assumptions of Theorem \ref{T1} be satisfied and let $F$ be given by \eqref{F}.
Then there are a nondecreasing function $h{:}\,\OT{\to}\,\R$ left-continuous on
$(0,T]$ and a continuous, increasing function $\omega:[0,\infty)\,{\to}\,\R$ with
$\omega(0)=0$ and such that
$F(\lambda,\cdot,\cdot)\in\mathcal{F}(\Omega\times\OT,h,\omega)$ for all
$\lambda\in\Lambda.$
\end{proposition}

\begin{proposition}\label{Schw2}
Let the assumptions of Theorem \ref{T1} be satisfied and let $F$ be given by \eqref{F}.
Then the integrals
\[
  \int_0^r D F(\lambda,x(\tau),t), \,\,
  \int_{0}^{r}f(\lambda,x(s),s)\,d s
  \mbox{\  and\ } \int_{0}^{r}g(x(s),s)\,d h(s)
\]
exist and the equality
\[
   \int_0^r D F(\lambda,x(\tau),t)
   =\int_{0}^{r}f(\lambda,x(s),s)\,d s+\int_{0}^{r}g(x(s),s)\,d h(s)
\]
holds for all $r\,{\in}\,\OT,$ $\lambda\,{\in}\,\Lambda$ and $x\,{\in}\,G\OT$
such that $x(s)\,{\in}\,\Omega$ for all $s\,{\in}\,\OT.$
\end{proposition}

\medskip

The correspondence between solutions of distributional differential equations and
generalized ordinary differential equations is clarified by the following theorem.
The proof follows easily from Proposition \ref{P1} and \cite[Theorem 4B.1]{Schw}
(cf. also \cite[Theorem 5.17]{Schwabik}).

\begin{theorem}\label{ponte}
Let the assumptions of Proposition \ref{Schw2} be satisfied. Then the couple
$(x,\lambda)\in G\OT\times\Lambda$ is a~solution of measure differential equation
\eqref{MDE} if and only if it is a~solution of the generalized ordinary
differential equation \thetag{1.2}.
\end{theorem}

\section{Bifurcation theory for Measure Differential Equations}

Let us turn our attention back to the periodic problem for the measure differential
equation
\begin{equation}\label{MDE-per}
   D x=f(\lambda,x,t)+g(x,t)\,D h,\quad x(0)=x(T).
\end{equation}
As in section 3, we will assume that conditions \eqref{A1}, \eqref{A2i} and
\eqref{A2ii} hold and $F:\Lambda\times\Omega\times\OT$ be given by \eqref{F}.
Then, by Proposition \ref{Schw}, there are a nondecreasing function
$h{:}\,\OT{\to}\,\R$ left-continuous on $(0,T]$ and a continuous, increasing
function $\omega:[0,\infty)\to\R$ with $\omega(0)=0$ and such that
$F(\lambda,\cdot,\cdot)\,{\in}\,\mathcal{F}(\Omega\times\OT,h,\omega)$ for
all $\lambda\in\Lambda.$ As a result, $F$ satisfies condition \eqref{B1}
from the previous section and, according to Theorem \ref{ponte}, the problems
\eqref{MDE-per} and
\begin{equation}\label{GDE-per}
   \dfrac{d x}{d \tau}=DF(\lambda,x,t),\quad x(0)=x(T),
\end{equation}
are equivalent.

Furthermore, we will assume also
\begin{equation}\label{A3}
\left\{\begin{array}{l}
  \mbox{$(x_0,\lambda)\in G\OT\times\Lambda$ is a~solution of \eqref{MDE-per}
  for any $\lambda\in\Lambda$ and there is a $\rho>0$}
\\[1mm]
  \mbox{such that $x(t)\in\Omega$ for all $t\in\OT$ and $x\in B(x_0,\rho).$}
\end{array}\right.
\end{equation}
Of course, then \eqref{B2} is true, as well.

Analogously to $\Phi,$ we define
\begin{equation}\label{operador1}\hskip-22mm
\left\{\begin{array}{l}\displaystyle
   \wt{\Phi}(\lambda,x)(t)
   =x(T)+\int_{0}^{t} f(\lambda,x(s),s)\,d s+\int_{0}^{t} g(x(s),s)\,d h(s)
  \\[3mm]
   \hskip20mm\mbox{for\ } \lambda\in\lambda,\, x\in B(x_0,\rho),\, t\in\OT.
\end{array}\right.
\end{equation}
By Proposition \ref{Schw}, we have
\begin{align*}
   &\wt{\Phi}(\lambda,x)(s)=x(T)+\int_0^s D F(\lambda,x(\tau),t)=\Phi(\lambda,x)(s)
 \\
   &\hskip20mm\mbox{for\ } s\in\OT, \lambda\in\Lambda \mbox{\ and\ } x\in B(x_0,\rho)
\end{align*}
and the following statement obviously holds.

\begin{proposition}\label{equiv1}
Let the assumptions of Theorem \ref{T1} be satisfied and let $F$ be given by \eqref{F}.
In addition, assume \eqref{A3} and let the operator $\wt{\Phi}$ be defined
by \eqref{operador1}. Then $\wt{\Phi}(\lambda,\cdot)$ maps $B(x_0,\rho)$ into
$G\OT$ for any $\lambda\in\Lambda.$ Moreover, problem \eqref{MDE-per} is
equivalent to finding couples $(x,\lambda)$ such that $x=\wt{\Phi}(\lambda,x),$
as well as to finding solutions $(x,\lambda)$ of \eqref{op-eq}.
\end{proposition}

Thus, it is natural to consider the bifurcation points of the periodic
problem \eqref{MDE-per} in the sense of Definition \ref{bif}.

\begin{definition}\label{biff}
Solution $(x_0,\lambda_0)\in G\OT\times\Lambda$ of \eqref{MDE-per} is said
to be a~{\em bifurcation point} of \eqref{MDE-per} if every neighborhood of
$(x_0,\lambda_0)$ in $B(x_0,\rho)\times\Lambda$ contains a~solution $(x,\lambda)$
of \eqref{MDE-per} such that $x\ne x_0.$
\end{definition}

Next statement follows from Theorem \ref{TheEx}.

\begin{corollary}\label{MeasEx}
Let the assumptions of Theorem \ref{T1} be satisfied. In addition, assume
\eqref{A3} and
\begin{equation}\label{A4}
\left\{\begin{array}{l}
\mbox{there is a $\gamma:\OT\to\R$ nondecreasing and such that for any $\eps>0$}
\\
\mbox{there is $\delta>0$ such that}
\\[2mm]\displaystyle
  \hskip10mm
  \Big\|\int_s^t [f(\lambda_2,x,r)-f(\lambda_1,x,r)]\,d r\Big\|_n
  <\eps\,|\gamma(t)-\gamma(s)|
 \\[3mm]
  \hskip10mm\mbox{for \ } x\in\Omega,\,t,s\in\OT
  \mbox{\ and\ } \lambda_1,\lambda_2\in\Lambda
  \mbox{\ such that\ }|\lambda_1-\lambda_2|<\delta.
\end{array}\right.
\end{equation}
Moreover, let the operator $\wt{\Phi}$ be defined by \eqref{operador1} and let
$[\lambda^*_1,\lambda^*_2]\subset\Lambda$ \ be such that
\begin{align}\label{ind2}
  &x_0 \mbox{\ is an isolated fixed point of the operators \ }
  \wt{\Phi}(\lambda^*_1,\cdot) \mbox{ \ and \ } \wt{\Phi}(\lambda^*_2,\cdot)
\\\noalign{\noindent\mbox{and}}\label{5.7}
  &\deg_{LS}(Id-\wt{\Phi}(\lambda^*_1,\cdot),B(x_0,\rho),0)
    \ne\deg_{LS}(Id-\wt{\Phi}(\lambda^*_2,\cdot),B(x_0,\rho),0).
\end{align}
Then there is $\lambda_0\in[\lambda^*_1,\lambda^*_2]$ such that $(x_0,\lambda_0)$
is a bifurcation point of \eqref{MDE-per}.
\end{corollary}
\begin{proof}
Recall that $F$ is given by \eqref{F} and hence, by Proposition \ref{equiv1},
the problems \eqref{GDEper} and \eqref{MDE-per} are then equivalent. Furthermore,
we already know that the assumptions \eqref{B1} and \eqref{B2} are satisfied.
Finally, our assumptions \eqref{A4}, \eqref{ind2} and \eqref{5.7} imply that
also all the remaining assumptions of Theorem \ref{TheEx} hold. This completes
the proof.
\end{proof}

\medskip

Our next wish is to find an explicit formula for the derivative of the
function $F$ given by \eqref{F}. This will be given by Proposition \ref{derr}.
In its proof we will need to interchange order of some iterated integrals.
This will be enabled by the following lemma inspired by Lemma 17.3.1 from
\cite{Hild} and valid for the Riemann-Stieltjes integrals, cf. also
\cite[Exercise II.19.3]{Hild}.

\begin{lemma}\label{iterated}
Let $-\infty<a<b<\infty,$ $-\infty<c<d<\infty,$  $g\in BV\ab,$
$h\in BV[c,d],$ and let $f:\ab\times [c,d]\to\mathcal{L}(\R^n)$ be bounded.
Moreover, let the integrals
\[
   G(s):=\int_a^b d g(\tau)\,f(\tau,s) \mbox{\ and\ }
   H(t):=\int_c^d f(t,\sigma)\,d h(\sigma)
\]
exist for all $s\in [c,d]$ and $t\in\ab.$

Then both the iterated integrals
\[
   \int_a^b d g(t)\Big(\int_c^d f(t,s)\,d h(s)\Big)
   \mbox{ \ and \ }
   \int_c^d \Big(\int_a^b [d g(t)]\,f(t,s)\Big) d h(s)
\]
exist and the equality
\[
   \int_a^b d g(t)\Big(\int_c^d f(t,s)\,d h(s)\Big)
   =\int_c^d \Big(\int_a^b [d g(t)]\,f(t,s)\Big) d h(s)
\]
holds.
\end{lemma}
\proof
Let $s\in [c,d]$ be given. Then to any $n\in\N$ we can choose a tagged
division $P_n=(\tau_j,[t_{j-1},t_j])_{j=1}^{\nu(P_n)}$ of $\ab$ such that
\[
 \left\|\sum_{j=1}^{\nu(P_n)}[g(t_j)-g(t_{j-1})]\,f(\tau_j,s)
 -G(s)\right\|_{n\times n}<\frac 1n.
\]
Hence, if we put
\[
   F_n(s)=\sum_{j=1}^{\nu(P_n)}[g(t_j)-g(t_{j-1})]\,f(\tau_j,s)
   \quad\mbox{for\ } n\in\N,
\]
then \ $\displaystyle\lim_{n\to\infty} F_n(s)=G(s).$ As $s$ was chosen
arbitrarily in $[c,d],$ it follows that
\begin{equation*}\label{limFn}
   \lim_{n\to\infty} F_n(s)=G(s) \quad\mbox{ \ for all \ } s\in[c,d].
\end{equation*}
Obviously, $|F_n(s)|\le K<\infty$ for all $n\in\N$ and $s\in [c,d],$
where
\[
  K=\sup\{\|f(t,s)\|_{n\times n}:(t,s)\in\ab\times[c,d]\}\,(\var_a^b\,g)
\]
Consequently, Bounded Convergence Theorem \cite[Theorem 6.8.13]{MST} yields
the existence of the integral $\displaystyle\int_c^d G(s)\,d h(s),$ while
\[
   \lim_{n\to\infty}\int_c^d F_n(s)\,d h(s)=\int_c^d G(s)\,d h(s).
\]
Analogously, we can show that also the integral
$\displaystyle\int_c^d d g(t)\,H(t)$ exists.

It remains to prove the equality
\[
   \int_c^d G(s)\,d h(s)=\int_a^b d g(t)\,H(t).
\]
To this aim, notice that for any $n\in\N$ a tagged division
$P_n=(\tau_j,[t_{j-1},t_j])_{j=1}^{\nu(P_n)}$ of $\ab$ from above we have
\begin{align*}
 &\sum_{j=1}^{\nu(P_n)}[g(t_j)-g(t_{j-1})]\,H(\tau_j)
 =\int_c^d\sum_{j=1}^{\nu(P_n)}[g(t_j)-g(t_{j-1})]\,f(\tau_j,s)\,d h(s)
\\
 &\quad=\int_c^d F_n(s)\,d h(s),
\end{align*}
while
\[
  \lim_{n\to\infty}\sum_{j=1}^{\nu(P_n)}[g(t_j)-g(t_{j-1})]\,H(\tau_j)
  =\int_a^b d g(t)\,H(t)
\]
and
\[
 \lim_{n\to\infty}\int_c^d F_n(s)\,d h(s)=\int_c^d G(s)\,d h(s).
\]
This completes the proof.
\endproof

In what follows the symbols $f\sp{\,\prime}_x(\lambda,x,t)$ and
$g\sp{\,\prime}_x(x,t)$ stand for real $n\times n$-matrices representing
respectively the total differentials of the functions $f$ and $g$ with respect
to $x$ at the points $(\lambda,x,t)$ or $(x,t),$ respectively, whenever
they have a sense.

\begin{proposition}\label{derr}
Let the assumptions of Theorem \ref{T1} be satisfied and let $F$ be given
by \eqref{F}. Moreover, let
\begin{equation}\label{D1i}\hskip-5mm
\left\{\begin{array}{l}
  \mbox{for every $(\lambda,x,t)\in\Lambda\times\Omega\times\OT$ the function $f$
  has a total differential $f\sp{\,\prime}_x$}
\\[1mm]
  \mbox{continuous with respect to $x\in\Omega$ for
  each $\lambda\in\Lambda$ and $t\in\OT$ and there is}
\\[1mm]
  \mbox{a Lebesgue integrable function $\Theta$ such that}
\\
  \hskip20mm
  \|f\sp{\,\prime}_x(\lambda,x,t)\|\le\Theta(t)
  \quad\mbox{for \ } (\lambda,x,t)\in\Lambda\times\Omega\times\OT
\end{array}\right.
\end{equation}
and
\begin{equation}\label{D1ii}\hskip-2mm
\left\{\begin{array}{l}
  \mbox{for every $(x,t)\in\Omega\times\OT$ the function $g$ has a total
  differential \ $g\sp{\,\prime}_x$ bounded}
\\[1mm]
  \mbox{on $\Omega\times\OT$ and continuous with respect to $x\in\Omega$
  for each $t\in\OT$ and}
\\[1mm]
  \mbox{there is $\Theta_u:\OT\to\R$ such that}
\\[2mm]
  \displaystyle\hskip20mm
  \int_0^T \Theta_u(s)\,d\,[\var_0^s\,h]<\infty \quad\mbox{and \ }
  \|g\sp{\,\prime}_x(x,t)\|\le\Theta_u(t)
\\
  \hskip25mm\mbox{for \ } (x,t)\in\Omega\times\OT.
\end{array}\right.
\end{equation}
Then, for every $(\lambda,x,t)\in\Lambda\times\Omega\times\OT$ the function
$F$ has a total differential $F\sp{\,\prime}_x(\lambda,x,t)$ and it is given by
\begin{equation}\label{deri}
  F\sp{\,\prime}_x(\lambda,x,t)=\int_0^t f\sp{\,\prime}_x(\lambda,x,s)\,d s
                           +\int_0^t g\sp{\,\prime}_x(x,s)\,d h(s)
   \quad\mbox{for all \ } (\lambda,x,t)\in\Lambda\times\Omega\times\OT.
\end{equation}
Moreover, $F\sp{\,\prime}_x(\lambda,\cdot,t)$  continuous with respect to
$x\in\Omega$ for any $(\lambda,t)\in\Lambda\times\OT.$
\end{proposition}
\begin{proof}
By the classical Leibniz Integral Rule, cf. e.g. \cite[V.39.1]{McSh}, we have
\[
  F\sp{\,\prime}_{1,x}(\lambda,x,t)
  =\int_0^t f\sp{\,\prime}_x(\lambda,x,s)\,d s
  \mbox{ \ for \ } (\lambda,x,t)\times\Lambda\times\Omega\times\OT.
\]
Analogously, the equality
\begin{equation}\label{deri2}
  F\sp{\,\prime}_{2,x}(x,t)=\int_0^t g\sp{\,\prime}_x(x,s)\,d h(s)
  \quad\mbox{for \ } (x,t)\in\Omega\times\OT
\end{equation}
could be essentially justified by the measure theory version of the Leibniz
Integral Rule, cf. e.g. \cite[Proposition 23.37]{yeh}. However, our setting
is little bit different. Hence, we feel that it would be honest to  give here
an independent proof. Let $(z,x,t)\in\R^n\times\Omega\times\OT$ be given, while
$x+z\in\Omega.$ Using the Mean Value Theorem (cf. \cite[Lemma 8.11]{Peter}),
we get
\begin{align*}
   \frac{F_2(x+\theta\,z,t)-F_2(x,t)}{\theta}
 &=\int_0^t
    \left[\frac{g\left(x+\theta z,s\right)-g(x,s)}{\theta}\right]\,d h(s)
\\
 &=\left(\int_0^t\left(
   \int_0^1\left[g\sp{\,\prime}_x(\alpha(x+\theta z)+(1-\alpha)\,x,s)\right]\,d\alpha
                 \right)d h(s)\right) z
\end{align*}
for any $\theta>0$ sufficiently small. By Lemma \ref{iterated} we have
\begin{multline*}
 \left(\int_0^t\left(
   \int_0^1\left[g\sp{\,\prime}_x(\alpha(x+\theta z)+(1-\alpha)\,x,s)\right]\,d\alpha
                 \right)d h(s)\right)
 \\
 =\left(\int_0^1\left(
   \int_0^t\left[g\sp{\,\prime}_x(\alpha(x+\theta z)+(1-\alpha)\,x,s)\right]\,d h(s)
                 \right)d\alpha\right).
\end{multline*}
Moreover, in view of \eqref{D1ii}, we get
\begin{align*}
   &\lim_{\theta\to 0+}g\sp{\,\prime}_x(\alpha(x+\theta\,z)+(1-\alpha)\,x,s)
                       =g\sp{\,\prime}_x(x,s)
  \\\noalign{\noindent\mbox{and}}
   &\left\|g\sp{\,\prime}_x(\alpha(x+\theta\,z)+(1-\alpha)\,x,s))\right\|\le\Theta_u(s)
   \mbox{ \ for all\ } (\alpha,s)\in[0,1]\times\OT.
\end{align*}
Let $h_1, h_2$ be functions nondecreasing on $\OT$ and such that $h=h_1-h_2$ on $\OT.$
Then, by Dominated Convergence Theorem (cf. \cite[Theorem 6.8.11]{MST}) we get
\[
  \lim_{\theta\to 0+}
                \int_0^t g\sp{\,\prime}_x(\alpha(x+\theta\,z)+(1-\alpha)\,x,s)\,d h_i(s)
               =\int_0^t g\sp{\,\prime}_x(x,s)\,d h_i(s)\in\mathcal{L}(\R^n)
\]
for $i=1,2.$ Consequently,
\begin{align*}
  \lim_{\theta\to 0+}\int_0^t g\sp{\,\prime}_x(\alpha(x+\theta\,z)+(1-\alpha)\,x,s)\,d h(s)
  &=\int_0^t g\sp{\,\prime}_x(x,s)\,d h(s)
 \\&\hskip-25mm
   =\int_0^t g\sp{\,\prime}_x(x,s)\,d h_1(s)-\int_0^t g\sp{\,\prime}_x(x,s)\,d h_2(s)
    \in\mathcal{L}(\R^n).
\end{align*}
Finally, as
\[
  \left\|\int_0^t g\sp{\,\prime}_x(x,s)\,d h(s)\right\|_{n\times n}
  \le\int_0^T\Theta_u(s)\,d\var_0^s\,h<\infty,
\]
using Dominated Convergence Theorem for Lebesgue integrals we complete the proof
of \eqref{deri}. The continuity of $F\sp{\,\prime}_x(\lambda,\cdot,t)$ then follows
readily thanks to the continuity assumptions contained in \eqref{D1i} and
\eqref{D1ii}.
\end{proof}

Next example is taken from \cite[Example 6.12]{JC}

\begin{example}\label{e1}
Consider the impulsive problem
\begin{equation}\label{ex}
   x'=\lambda\,b(t)\,x+c(t)\,x^2,\quad \Delta^+x(\tfrac 12)=x^2(\tfrac 12),
   \quad x(0)=x(1)
\end{equation}
with \ $b,c\in L^1[0,1]$ \ and \ $\displaystyle\int_0^1 b(s)\,d s\ne 0,$
\ i.e.,
\[
  x(t)=x(1)+\int_0^t f(\lambda,x(s),s)\,d s+\int_0^t g(x(s),s)\,d h(s)
\]
where  $f(\lambda,x,s)=\lambda\,b(s)\,x+c(s)\,x^2,$ $g(x,s)=x^2,$ \
$h(s)=\chi_{(\tfrac 12,1]}(s).$

Obviously, $x_0(t)\equiv 0$ is a solution of \eqref{ex} for all $\lambda.$
Linearization at $x_0$ yields
\[
     z'=\lambda\,b(t)\,z,\quad z(0)=z(1)
     \Leftrightarrow
     \left\{\begin{array}{l}
        \lambda=0 \,\land\,  z\equiv const,
       \\
        \lambda\ne 0 \,\land\, z\equiv 0.
     \end{array}\right.
\]
One can verify, cf. \cite[Example 6.12]{JC} that the assumptions of Corollary \ref{MeasEx}
are satisfied. In particular, there are \ $\lambda^*_1,\lambda^*_2$ such that
$-1<\lambda^*_1<0<\lambda^*_2<1$ and
\[
     \mbox{ind}_{LS}(Id-\wt{\Phi}(-\delta,0))
     =-\,\mbox{deg}_{LS}(Id-\wt{\Phi}(\delta,0))
\]
for any \ $\delta>0$ sufficiently small. Thus, by Corollary \ref{MeasEx}, there exist
a $\delta^*>0$ such that for any $\delta\in(0,\delta^*)$ there is $\lambda_0\in(-\delta,\delta)$
such that $(\lambda_0,0)$ is a bifurcation point of \eqref{ex}.
\end{example}

\medskip

Proposition \ref{Schw2} can be obviously modified to matrix valued
function. Therefore, we can state the following assertion.

\begin{proposition}\label{3.9der}
Let the assumptions of Proposition \ref{derr} be satisfied.
Then all the integrals
\[
  \int_0^r D F\sp{\,\prime}_x(\lambda,x(\tau),t),\,\,
  \int_{0}^{r}f\sp{\,\prime}_x(\lambda,x(s),s)\,d s,\,\,
  \int_{0}^{r}g\sp{\,\prime}_x(x(s),s)\,d h(s)
\]
exist and the equality
\begin{equation}\label{DF}
   \int_0^r D[F\sp{\,\prime}_x(\lambda,x(\tau),t)]
   =\int_{0}^{r}f\sp{\,\prime}_x(\lambda,x(s),s)\,d s
     +\int_{0}^{r}g\sp{\,\prime}_x(x(s),s)\,d h(s)
\end{equation}
holds for all $r\in\OT,$ $\lambda\in\Lambda$ and $x\in G\OT$ such that
$x(s)\in\Omega$ for all $s\in\OT.$
\end{proposition}

Next result characterizes the derivative $\wt{\Phi}\sp{\,\prime}_x$ of
the operator $\wt{\Phi}$ given by \eqref{operador1}.

\begin{proposition}\label{semi}
Let the assumptions of Proposition \ref{derr} be satisfied. Then, for
given $(\lambda,x)\in\Lambda\times B(x_0,\rho),$ the derivative
$\wt{\Phi}\sp{\,\prime}_x(\lambda,x)$ of $\wt{\Phi}(\lambda,\cdot)$
at $x$ is given by
\begin{equation}\label{fi'}
    \left(\wt\Phi\sp{\,\prime}_x(\lambda,x)z\right)(t)
    =z(T)+\int_0^t f\sp{\,\prime}_x(x(s),s)\,z(s)\,d \tau
         +\int_0^t g\sp{\,\prime}_x(\lambda,x(s),s)\,z(s)\,du(s)
\end{equation}
for all $z\in G\OT$ and $t\in [0,T].$
\end{proposition}
\proof
Using Proposition \ref{3.9der}, where $u$ need not be monotone, and
analogously to the proof of item 2 of Lemma 5.1 in \cite{Ant}, we
can verify the equality
\begin{equation}\label{equal}
 \int_0^r D[F\sp{\,\prime}_{x}(\lambda,x(\tau),t)\,z(\tau)]
   =\int_0^r f'_x(x(\tau),\tau)\,z(\tau)\,d\tau
   +\int_0^t g'_x(\lambda,x(\tau),\tau)\,z(\tau)\,d h(\tau),
\end{equation}
for every $t\in\OT$, $\lambda\in\Lambda$ and  $x,z\in G\OT$ such that
$x(s)\in\Omega$ for all $s\in\OT.$ Indeed, by Proposition \ref{3.9der},
relation \eqref{DF} is true for every $x\in G\OT.$ Now, let
$[\alpha,\beta]\subset\OT,$ $z\in G\OT$ and $z(t)=\wt{z}\in\R^n$ for
$t\in(\alpha,\beta).$ Then by \eqref{DF}, Lemma \ref{Ant} and Hake Theorem
(cf. e.g. \cite[Theorem 6.5.6]{MST}) we compute
\begin{align*}
   &\int_\alpha^\beta D[F\sp{\,\prime}_x(\lambda,x(\tau),t)\,z(\tau)]
  \\&\,=
    \lim_{\delta\to 0+}\Big(
       \int_{\alpha+\delta}^{\beta-\delta}
               D[F\sp{\,\prime}_x(\lambda,x(\tau),t)]\,\wt{z}
       +\int_\alpha^{\alpha+\delta} D[F\sp{\,\prime}_x(\lambda,x(\tau),t)\,z(\tau)]
  \\&\hskip10mm
       +\int_{\beta-\delta}^\beta D[F\sp{\,\prime}_x(\lambda,x(\tau),t)\,z(\tau)]\Big)
  \\&\,=
    \int_\alpha^\beta f\sp{\,\prime}_x(\lambda,x(\tau),\tau)\,z(\tau)\,d \tau
    +\lim_{\delta\to 0+}\Big(\int_{\alpha+\delta}^{\beta-\delta}
                             g\sp{\,\prime}_x(x(\tau),\tau)\,z(\tau)\,d h(\tau)
  \\&\hskip3mm+
     (F\sp{\,\prime}_x(\lambda,x(\alpha),\alpha+\delta){-}
      F\sp{\,\prime}_x(\lambda,x(\alpha),\alpha))\,z(\alpha)
    +(F\sp{\,\prime}_x(\lambda,x(\beta),\beta){-}
      F\sp{\,\prime}_x(\lambda,x(\beta),\beta-\delta))\,z(\beta)\Big)
  \\&\,=
    \int_\alpha^\beta f\sp{\,\prime}_x(\lambda,x(\tau),\tau)\,z(\tau)\,d \tau
    +\lim_{\delta\to 0+}\Big(\int_{\alpha+\delta}^{\beta-\delta}
                             g\sp{\,\prime}_x(x(\tau),\tau)\,z(\tau)\,d h(\tau)
  \\&\hskip3mm+
     g\sp{\,\prime}_x(x(\alpha),\alpha)\,z(\alpha)\,(u(\alpha+\delta)-u(\alpha))
    +g\sp{\,\prime}_x(x(\beta),\beta)\,z(\beta)\,(u(\beta)-u(\beta-\delta))\Big)
  \\&\,=
      \int_\alpha^\beta f\sp{\,\prime}_x(\lambda,x(\tau),\tau)\,z(\tau)\,d \tau
    +\int_\alpha^\beta g\sp{\,\prime}_x(x(\tau),\tau)\,z(\tau)\,d h(\tau).
\end{align*}
Having in mind that every regulated function is a uniform limit of step
(piece-wise constant) functions, we complete the proof by means of the
Uniform Convergence Theorem (cf. e.g. \cite[Theorem 6.8.2]{MST}).
\endproof
\medskip

Now, we show that $(\lambda_0,x_0)$ is not a~bifurcation point of the
operator equation $\wt{\Phi}(\lambda,x)\,{=}\,x$ whenever
$Id-\wt{\Phi}\sp{\,\prime}_{x}(\lambda_0,x_0)$ is an isomorphism on $G\OT.$

\begin{theorem}\label{T59}
Let the assumptions of Proposition \ref{derr} be satisfied. Moreover, assume
that \eqref{A3} and \eqref{A4} hold and
\begin{equation}\label{D4}\hskip-8mm
\left\{\begin{array}{l}
\mbox{there is a~nondecreasing function $\wt{\gamma}:\OT\to\R$ such that
for any $\eps>0$}
\\
\mbox{there is a~$\delta>0$ such that}
\\[2mm]\displaystyle\quad
 \left\|
 \int_s^t [f\sp{\,\prime}_x(\lambda_1,x,r)-f\sp{\,\prime}_x(\lambda_2,y,r)]\,d r
 +\int_s^t [g\sp{\,\prime}_x(x,r)-g\sp{\,\prime}_x(y,r)]\,d h(r)
 \right\|_{n\times n}
\\[4mm]
 \qquad<\eps\,|\wt{\gamma}(t)-\wt{\gamma}(s)|
\\[2mm]
\mbox{for all $t,s\in\OT$ and all $x,y\in\Omega, \lambda_1,\lambda_2\in\Lambda$
      satisfying}
\\[1mm]\qquad
  |\lambda_1{-}\lambda_2|+\|x\,{-}\,y\|_n<\delta.
\end{array}\right.
\end{equation}

Let the operator $\wt\Phi$ be defined by \eqref{operador1} and let
$\lambda_0\in\Lambda$ be given. Let
$Id-\wt{\Phi}\sp{\,\prime}_{x}(\lambda_0,x_0)$ be an isomorphism of
$G\OT$ onto $G\OT.$ Then there is $\delta>0$ such that $(x,\lambda)$
is not a~bifurcation point of the equation $\wt{\Phi}(\lambda,x)=x$
whenever $\|x-x_0\|_\infty+|\lambda-\lambda_0|<\delta.$
\end{theorem}
\begin{proof}
Recall that, in addition to \eqref{A3}, \eqref{A4} and \eqref{D4}, we assume
that, as in Proposition \ref{derr}, the conditions \eqref{A1}, \eqref{A2i},
\eqref{A2ii}, \eqref{D1i} and \eqref{D1ii} hold, as well.
Let $F$ be given by \eqref{F}. Then, by Proposition \ref{derr}, its derivative
with respect to $x$ is given by \eqref{deri}, i.e.
\[
F\sp{\,\prime}_x(\lambda,x,t)=\int_0^t f\sp{\,\prime}_x(\lambda,x,s)\,d s
                           +\int_0^t g\sp{\,\prime}_x(x,s)\,d h(s)
   \quad\mbox{for all \ } (\lambda,x,t)\in\Lambda\times\Omega\times\OT
\]
for $z\in G\OT$ and $t\in\OT.$

Furthermore, by Proposition \ref{semi}, the derivative with respect to $x$
of $\wt{\Phi}(\lambda,\cdot)$ is given by \eqref{fi'}, i.e
\[
   \wt\Phi\sp{\,\prime}_x(\lambda,x)z(t)
   =z(T)+\int_0^t f\sp{\,\prime}_x(x(s),s)\,z(s)\,d \tau
        +\int_0^t g\sp{\,\prime}_x(\lambda,x(s),s)\,z(s)\,du(s)
\]
for all $(\lambda,x)\in\Lambda\times B(x_0,\rho),$ $z\in G\OT$ and
$t\in [0,T].$ Moreover, by relation \eqref{equal} from the proof of
the same proposition we have
\begin{equation}\label{P*P}\hskip-20mm
\left\{\begin{array}{l}
   \left(\wt{\Phi}\sp{\,\prime}_x(\lambda,x)\,z\right)(t)=
   \left(\Phi\sp{\,\prime}_x(\lambda,x)\,z\right)(t)
  \\[4mm]
   \qquad\mbox{for\ } (\lambda,x)\in\Lambda\times B(x_0,\rho), z\in G\OT
   \mbox{\ and\ } t\in [0,T],
\end{array}\right.
\end{equation}
where $\Phi$ and $\Phi\sp{\,\prime}_{x}$ are respectively given by
\eqref{operador} and \eqref{de}.

\medskip

Now, suppose that $Id-\wt\Phi\sp{\,\prime}_{x}(\lambda_0,x_0):G\OT\to G\OT$
is an isomorphism.  Then, due to \eqref{P*P}, the mapping
$Id-\Phi\sp{\,\prime}_x(\lambda_0,x_0):G\OT\to G\OT$ is an isomorphism,
as well.

\medskip

We want to apply Theorem \ref{der}. To this aim we need to verify that all
its assumptions, i.e. \eqref{B1}, \eqref{B2}, \eqref{B3}, \eqref{C1} and
\eqref{C3} are satisfied.

First, notice that the periodic problem for the equation \eqref{MDE}
is by Theorem \ref{T1} equivalent to the periodic problem for
the integral equation \eqref{SIE}. Moreover, by Proposition
\ref{Schw} there are a nondecreasing function $h:\OT\to\R$
left-continuous on $(0,T]$ and a continuous, increasing function
$\omega:[0,\infty)\to\R$ with $\omega(0)=0$ and such that
$F(\lambda,\cdot,\cdot)\in\mathcal{F}(\Omega\times\OT,h,\omega)$
for all $\lambda\in\Lambda.$ In particular, \eqref{B1} is satisfied.
Moreover, Theorem \ref{ponte} implies that the periodic problem
\eqref{GDEper} is equivalent with the periodic problem for the
equation \eqref{SIE} and, hence, $F$ satisfies also \eqref{B2}.

Second, from \eqref{D1i}, \eqref{D1ii} and \eqref{deri} it follows
immediately that \eqref{C11} is also true if we put
\[
  \wt{h}(t)=\int_0^t\Theta(r)\,d r+\int_0^t\Theta_u(r)\,d[\var_0^r u].
\]

Finally, it remains to show that \eqref{C12} is satisfied, as well.
By \eqref{deri} and \eqref{D4}, there is a~nondecreasing function
$\wt{\gamma}:\OT\to\R$ such that for any $\eps>0$ there is a~$\delta>0$
such that
\begin{align*}
&\|F\sp{\,\prime}_{x}(\lambda_1,x,t)-F\sp{\,\prime}_{x}(\lambda_2,y,t)
-F\sp{\,\prime}_{x}(\lambda_1,x,s)+F\sp{\,\prime}_{x}(\lambda_2,y,s)\|_{n\times n}
\\&\quad
   =\left\|
      \int_s^t[f\sp{\,\prime}_x(\lambda_1,x,r)
                -f\sp{\,\prime}_x(\lambda_2,y,r)]\,d r
     +\int_s^t[g\sp{\,\prime}_x(x,r)
                 -g\sp{\,\prime}_x(y,r)]\,du(r)
   \right\|_{n\times n}\hskip-7mm
\\&\quad
   <\eps\,|\wt{\gamma}(t)-\wt{\gamma}(s)|
\\[4mm]
 &\mbox{for all\ }t,s\in\OT \mbox{\ and all\ }
  x,y\in\Omega,\lambda_1,\lambda_2\in\Lambda \mbox{\ such that\ }
  |\lambda_1\,{-}\,\lambda_2|+\|x\,{-}\,y\|_n<\delta.
\end{align*}
This means that \eqref{C12} and \eqref{C3} are true when we take
$\lambda_1=\lambda_2$ and $x=y$ in the last inequality. Moreover, by
\eqref{A4}, we obtain that also \eqref{B3} is satisfied. Thus, all
the hypotheses of Theorem \ref{der} are satisfied. Therefore, $(\lambda_0,x_0)$
is not a~bifurcation point of the equation $\wt{\Phi}(\lambda,x)\,{=}\,x$
and there is $\delta>0$ such that $(x,\lambda)$ is not a~bifurcation point
of this equation whenever $\|x-x_0\|_\infty+|\lambda-\lambda_0|<\delta.$
This completes the proof.
\end{proof}

Finally, analogously to Theorem \ref{NC} we can state a~necessary condition
for the existence of the bifurcation point to the problem \eqref{MDE-per} in
the form related to the Fredholm type alternative.

\begin{theorem}\label{last}
Let the assumptions of Theorem \ref{T59} be satisfied and let the couple
$(\lambda_0,x_0)\,{\in}\,\Lambda\times\Omega$ be a~bifurcation point
of problem \eqref{MDE-per}. Then then there exists \ $q\in G\OT$ such that
the equation
\[
 z(r)-z(T)\,{-}\int_0^t f\sp{\,\prime}_x(\lambda_0,x_0,\tau)\,z(\tau)\,d\tau
          \,{-}\int_0^t g\sp{\,\prime}_x(x_0,\tau)\,z(\tau)\,du(\tau)=q(t),
 \quad\mbox{for \ } r\in\OT
\end{equation*}
has no solution in $G\OT$ and the corresponding homogeneous equation
\begin{align*}
  &z(r)-z(T)-\int_0^t f\sp{\,\prime}_x(\lambda_0,x_0,\tau)\,z(\tau)\,d\tau
           -\int_0^t g\sp{\,\prime}_x(x_0,\tau)\,z(\tau)\,du(\tau)=0
  \quad\mbox{for \ } r\in\OT
\end{align*}
possesses at least one nontrivial solution in $G\OT.$
\end{theorem}
\begin{proof}
Suppose $(\lambda_0,x_0)$ is a~bifurcation point of \eqref{MDE-per}, i.e of the
equation
\[
   \wt{\Phi}(\lambda,x)\,{=}\,x
\]
with $\wt{\Phi}$ given by \eqref{operador1}. Then, by Proposition \ref{Schw},
$(\lambda_0,x_0)$ is also a~bifurcation point of the equation
$\Phi(\lambda,x)\,{=}\,x,$ where $\Phi$ is given by \eqref{operador}. Our
statement follows by Theorem \ref{NC}.
\end{proof}

\begin{example}
Consider the problem \eqref{ex} from Example \ref{e1}, i.e.
\begin{equation}\tag{5.12}
  x(t)=x(1)+\int_0^t f(\lambda,x(s),s)\,d s+\int_0^t g(x(s),s)\,d h(s)
\end{equation}
where  $f(\lambda,x,s)=\lambda\,b(s)\,x+c(s)\,x^2,$ $g(x,s)=x^2,$ \
$h(s)=\chi_{(\tfrac 12,1]}(s)$ and \ $b,c\in L^1[0,1]$ \ and
\ $\displaystyle\int_0^1 b(s)\,d s\ne 0.$
We can verify that $f, g, h$ satisfy the assumptions of Theorem \ref{last}.
Furthermore, we know that $x_0(t)\equiv 0$ is a solution of \eqref{ex} for
all $\lambda\in\Lambda$ and that for any $\delta>0$ sufficiently small we
can find $\lambda\in(-\delta,\delta)$ such that $(\lambda,0)$ is a bifurcation
point of \thetag{5.12} such that $\lambda\in(-\delta,\delta).$ In other words,
it could happen that there is a line segment $J=(-\wt{\delta},\wt{\delta})$
such that any couple $(\lambda,0),$ with $\lambda-n J$ is a bifurcation
point of \thetag{5.12}.

On the other hand, $z'=\lambda\,b(t)\,z,\quad z(0)=z(1)$ is the corresponding
linearized problem at $x_0$ and as $z$ is its solution if and only if $\lambda=0$
and $z$ is constant or $\lambda\ne 0$ and $z\equiv 0,$ Theorem \ref{last} implies
that $(\lambda,0)$ can not be a bifurcation point of \thetag{5.12} whenever
$\lambda\,{\ne}\,0.$ Consequently, $(0,0)$ is the only bifurcation point of
\thetag{5.12}.
\end{example}

For further example the following special case of the result by A. Lomtatidze
(cf. \cite[Theorem 11.1 and Remark 0.5]{AL}) will be useful.

\begin{proposition}\label{Ba}
Let $q:[0,T]\to\R$ be continuous and such that
\begin{gather*}
   \int_0^Tq_-(s)\,d s>0 \quad\mbox{and}\quad \int_0^Tq_+(s)\,d s>0
\\\noalign{\noindent\mbox{where, as usual,}}
   q_+(t):=\max\{q(t),0\} \mbox{ \ and \ } q_-(t):=-\min\{q(t),0\}
   \mbox{ \ for\ } t\in\OT.
\end{gather*}
Further, assume that
\begin{equation}\label{bacho}
  \int_0^Tq_-(s)\,d s<\left(1\,{-}\,\frac{\pi}{2}\int_0^Tq_-(s)\,d s\right)
  \int_0^Tq_+(s)\,d s
\quad\mbox{and \ }
  \int_0^Tq_-(s)\,d s<\dfrac{2}{\pi}.
\end{equation}
Then the equation \ $y''+q(t)\,y=0$ has only trivial $T$-periodic solution.
\end{proposition}

\medskip

\begin{example}\label{lieb}
By Example (4.2) in \cite{CITZ} (cf. also \cite[Remark 3.1]{CS})
the function
\[
   u(t)=u_0(t)=(2+\cos t)^3
\]
is a solution of the problem
\[
  u''(t)=(6.6-5.7\,\cos t-9\,\cos^2 t)\,u^{1/3}-0.3\,u^{2/3},\,\,
  u(0)=u(2\,\pi),\, u'(0)=u'(2\,\pi)
\]
related to the Liebau valveless pumping phenomena. Since $u_0(\pi)=1,$
\begin{align*}
  &2\,(u_0(\pi))^3-(u_0(\pi))^2-4\,u_0(\pi)+3=0
 \\\noalign{\noindent\mbox{and}}
  &(2+\cos\,t)\,u'_0(t)+3\,(\sin t)\,u_0(t)=0
   \quad\mbox{for all\ } t\in[0,2 \pi],
\end{align*}
$u=u_0$ clearly solves also the parameterized impulsive problem
\begin{equation}\label{liebau}
\left\{\begin{array}{l}
   u''=\lambda\left((2\,{+}\,\cos t)\,u'{+}\,3\,(\sin t)\,u\right)
   +(6.6{-}\,5.7\,\cos t\,{-}\,9\,\cos^2 t)\,u^{1/3}{-}\,0.3\,u^{2/3},
  \\[2mm]
   \Delta^+u(\pi)=2\,(u(\pi))^3-(u(\pi))^2-4\,u(\pi)+3,\quad
    u(0)=u(2 \pi),\, u'(0)=u'(2 \pi)
\end{array}\right.
\end{equation}
for all $\lambda\in\R.$

To prove that the couple $(u_0,0)$ is not a bifurcation point of \eqref{liebau}, we want
to apply Theorem \ref{last}. To this aim, we rewrite the problem \eqref{liebau} as the
integral system
\begin{align}\label{5.20}
   &x(t)=x(2 \pi)+\int_0^t f(\lambda,x(s),s)\, d s + \int_0^t g(x(s),s)\,d h(s),
\\[2mm]\noalign{\noindent\mbox{where}}\nonumber
   x_1=u,\, &x_2=u',\, x=\begin{pmatrix}x_1\\x_2\end{pmatrix},
\\[2mm]\nonumber
   f(\lambda,x,t)&=\begin{pmatrix}
        x_2
      \\[1mm]
\lambda\left((2\,{+}\,\cos t)\,x_2{+}\,3\,(\sin t)\,x_1\right)+R(t)\,x_1^{1/3}{-}\,0.3\,x_1^{2/3},
     \end{pmatrix}
\\[2mm]\nonumber
   g(x,t)&=\begin{pmatrix} 2\,x_1^3-x_1^2-4\,x_1+3 \\[1mm] 0 \end{pmatrix},\quad
    h(t)=\chi_{(\pi,2\pi]}(t)
\\\noalign{\noindent\mbox{and}}\nonumber
   R(t)&=6.6-5.7\,\cos t-9\,\cos^2 t.
\end{align}
Obviously, $x_0=\begin{pmatrix}u_0\\u\sp{\,\prime}_0\end{pmatrix}$ is a solution
to \eqref{5.20} for all $\lambda\in\R.$ Choose $\Omega=(0.5,28)\times(-20,20),$
$\Lambda=(-1,1)$ and $\rho=0.25.$ Then, it is possible to verify that $x(t)\in\Omega$
for all $x\in B(x_0,\rho)$ and we can conclude that $f,\, g$ and $h$ satisfy conditions
\eqref{A1} and \eqref{A3}. Moreover, it is easy to verify that the assumption \eqref{A2i},
\eqref{A2ii}, \eqref{A3}, \eqref{D1i} and \eqref{D1ii} are satisfied, as well.

Next we show that also \eqref{D4} holds. To this aim, consider the expression
\[
 \Delta(t,s,x,y,\lambda_1,\lambda_2)
 :=\int_s^t [f\sp{\,\prime}_x(\lambda_1,x,r)-f\sp{\,\prime}_x(\lambda_2,y,r)]\,d r
 +\int_s^t [g\sp{\,\prime}_x(x,r)-g\sp{\,\prime}_x(y,r)]\,d h(r),
\]
where $0\le s<t\le 2 \pi,$ $x,y\in\Omega,$ and $\lambda_1,\lambda_2\in\Lambda.$
As
\begin{equation}\label{5.21}
\left\{\begin{array}{l}\displaystyle
   f\sp{\,\prime}_x(\lambda,x,t)=
     \begin{pmatrix}
       0 & , & 1
      \\[2mm]]
       \lambda\,3\,\sin\,t{+}\,\tfrac 13\,R(t)\,x_1^{-\tfrac 23}
          \,{-}\,0.2\,x_1^{-\tfrac 13}
       & , &\lambda\,(2+\cos\,t)\,
     \end{pmatrix},
 \\[10mm]\displaystyle\quad
     g\sp{\,\prime}_x(x,t)=
     \begin{pmatrix}6\,x_1^2-2\,x_1-4 & , & 0
       \\[2mm]
                    0 &, & 0
     \end{pmatrix}
\end{array}\right.
\end{equation}
for $\lambda\in\Lambda,$ $x\in\Omega$ and $t\in[0,2 \pi],$ it is not difficult
to justify the inequality
\begin{align*}
  \|\Delta(t,s,x,y,\lambda_1,\lambda_2)\|_{2{\times}2}
  &\le 5\left(|\lambda_1-\lambda_2|+
     |x_1^{-\tfrac 13}-y_1^{-\tfrac 13}|+|x_1^{-\tfrac 23}-y_1^{-\tfrac 23}|\right)(t-s)
\\&\quad
 +(6\,|x^2_1-y^2_1|+2\,|x_1-y_1|)\,(h(t)-h(s))
\end{align*}
for $0\le s<t\le 2 \pi,$ $x,y\in\Omega,$ and $\lambda_1,\lambda_2\in\Lambda.$ Now, having
in mind that the functions $x^2,$ $x^{-\tfrac 13}$ and $x^{-\tfrac 23}$ are uniformly
continuous on $[0.5,28],$ it is already easy to verify that the assumption \eqref{D4} will
be satisfied if we put $\wt{\gamma}(t)=t+h(t).$

Finally, since
\begin{align*}
   \|\int_s^t [f(\lambda_1,x,r)-f(\lambda_2,x,r)]\,d r\|_2
  &\le |\lambda_1-\lambda_2|\, \left[
    3\, (|x_2|+ |x_1|) \right](t-s)
\end{align*}
for $0\le s<t\le 2 \pi,$ $x\in\Omega,$ and $\lambda_1,\lambda_2\in\Lambda,$
we can see that assumption \eqref{A4} will be satisfied with $\gamma(t)=t.$

The linearization of \eqref{5.20} around $(x_0,0)$ is
\begin{equation}\label{linear}
   z(t)=z(2 \pi)+\int_0^t f\sp{\,\prime}_x(0,x_0(t),r)\,z(r)\,d r
          + g\sp{\,\prime}_x(x_0(\pi),\pi)\,z(\pi)\,\chi_{(0,\pi]}(t)
   \quad\mbox{for\ } t\in[0,2 \pi].
\end{equation}
Inserting $\lambda=0$ and $x_0=\begin{pmatrix}u_0\\ u'_0\end{pmatrix}$ into \eqref{5.21},
we get
\begin{align*}
     f\sp{\,\prime}_x(0,x_0(t),t)&=
     \begin{pmatrix}0 & , & 1
        \\[2mm]
           \tfrac 13\,R(t)\,(u_0(t))^{-\tfrac 23}\,{-}\,0.2\,(u_0(t))^{-\tfrac 13}
                & , & 0\hskip1mm
     \end{pmatrix}
   \\&=
     \begin{pmatrix}0 & , & 1
        \\[2mm]\dfrac{3\,(6-7\,\cos t-10\,\cos^2 t)}{10\,(2+\cos t)^2} & , & 0\end{pmatrix}
\\\noalign{\noindent\mbox{and}}
   g\sp{\,\prime}_x(x_0(\pi),\pi)&=
      \begin{pmatrix}6\,(u_0(\pi))^2-2\,u_0(\pi)-4 & , & 0\\[2mm] 0 &, & 0\end{pmatrix}
     =\begin{pmatrix} 0 &, & 0\\[2mm] 0 &, & 0\end{pmatrix}.
\end{align*}
This means that \eqref{linear} reduces to the second order periodic problem
\begin{align}\label{lin}
  z''&=q(t)\,z,\quad z(0)=z(2 \pi),\, z'(0)=z'(2 \pi),
 \\\noalign{\noindent\mbox{where}}\nonumber
  q(t)&=\dfrac{3\,(6-7\,\cos t-10\,\cos^2 t)}{10\,(2+\cos t)^2}
    \quad\mbox{for \ } t\in[0,2 \pi].
\end{align}
One can compute:
\[
   \int_0^{2 \pi}q_-(s)\,d s=
   2\pi-\dfrac{5\,(6+59\,\arctan 1/3)}{5\,\sqrt{3}}\approx 0.513543.
\]
In particular,
\[
  0<1-\dfrac{\pi}2\int_0^{2 \pi}q_-(s)\,d s\approx 0.193328.
\]
Furthermore,
\begin{align*}
  &\int_0^{2 \pi}q_+(s)\,d s=
   \dfrac 1{15}\,\left((59\,\sqrt{3}-60)\,\pi-2\sqrt{3}\,(6+\arctan 1/3)\right)
  \approx 3.06682,
\\\noalign{\noindent\mbox{and}}
  &\dfrac2{\pi}\approx 0.63662>\left(1-\frac{\pi}{2}\,\int_0^{2 \pi}q_-(s)\,d s\right)
  \left(\int_0^{2 \pi} q_+(s)\,d s\right)\approx 0.592902
\\
  &\hskip20mm>0.513543\approx\int_0^{2 \pi}q_-(s)\,d s.
\end{align*}
Consequently,  Proposition \ref{Ba} implies that the linear problem \eqref{lin}
possesses only the trivial solution. Thus, by Theorems \ref{T59} and \ref{last}, we
conclude that there is a $\delta>0$ such that $(x,\lambda)$ is not a bifurcation point
of \eqref{liebau}  whenever $|\lambda|+\|x-x_0\|_\infty<\delta.$ In particular,
the couple $(x_0,0)$ can not be a bifurcation point of \eqref{liebau}.

Note that the validity of the assumptions of Theorems \ref{T59} and \ref{last} for
the model worked out in this example can also be verified using Corollary 2.1 in \cite{HaTo}.

Some computations in this example were made with the help of the software system
Mathematica.
\end{example}

\section*{Acknowledgements}
We would like to thank to Ji\v{r}\'{\i} \v{S}remr and Robert Hakl for their valuable
help with Example \ref{lieb}.

\smallskip

C.~Mesquita was supported by CAPES n. 88881.187 960/2018-01 and by the Institutional
Research Plan RVO 6798584 of the Czech Academy of Sciences.

M.~Tvrd\'y was supported the Institutional Research Plan RVO 6798584 of the Czech
Academy of Sciences.

\end{document}